\documentclass{article}
\usepackage{arxiv}
\usepackage[utf8]{inputenc} 
\usepackage[T1]{fontenc}    
\usepackage{hyperref}       
\usepackage{url}            
\usepackage{booktabs}       
\usepackage{amsfonts}       
\usepackage{nicefrac}       
\usepackage{microtype}      
\usepackage{graphicx}
\usepackage[backend=biber,style=alphabetic,sorting=ynt]{biblatex}
\usepackage{csquotes}
\usepackage{doi}
\usepackage{tabularx}
\usepackage{lipsum}
\usepackage{wrapfig}
\usepackage{listings}
\usepackage{mathrsfs}
\usepackage{amsmath}
\usepackage{amssymb}
\usepackage{mathabx}
\usepackage{bbm}
\usepackage{bm}
\usepackage{esvect}
\usepackage{upgreek}
\usepackage{cases}

\usepackage{amsthm}
\theoremstyle{definition}
\newtheorem{definition}{Definition}[section]
\newtheorem{example}{Example}[section]
\newtheorem{proposition}{Proposition}[section]
\newtheorem{remark}{Remark}[section]

\newtheorem{theorem}{Theorem}[section]
\newtheorem{lemma}{Lemma}[section]
\newtheorem{corollary}{Corollary}[section]

\usepackage{algorithm} 
\usepackage{algpseudocode} 
\usepackage{arydshln}

\usepackage{hyperref}

\title{Comparing discounted and average-cost Markov Decision Processes: \emph{a statistical significance perspective}.}

\author{Dylan Solms\\
	Department of Decision Sciences\\
	University of South Africa\\
	South Africa \\
	\texttt{62652257@mylife.unisa.ac.za} \\
}

\renewcommand{\headeright}{Discussion}
\renewcommand{\undertitle}{Discussion}
\renewcommand{\shorttitle}{Markov Decision Processes}

\hypersetup{
pdftitle={Comparing discounted and average-cost Markov Decision Processes: \emph{a statistical significance perspective}.},
pdfsubject={Markov Decision Process},
pdfauthor={Dylan Solms},
pdfkeywords={Markov Decision Process,Blackwell optimality,Queuing theory,Hypothesis testing,Simulation},
}

\begin{document}
\maketitle

\begin{abstract}
Optimal Markov Decision Process policies for problems with finite state and action space are identified through a partial ordering by comparing the value function across states. This is referred to as state-based optimality. This paper identifies when such optimality guarantees some form of system-based optimality as measured by a scalar. Four such system-based metrics are introduced. Uni-variate empirical distributions of these metrics are obtained through simulation as to assess whether theoretically optimal policies provide a statistically significant advantage. This has been conducted using a Student's $t$-test, Welch's $t$-test and a Mann-Whitney $U$-test. The proposed method is applied to a common problem in queuing theory: admission control.
\end{abstract}

\tableofcontents

\newpage

\section{Introduction}

This paper has been motivated by the following scenario. Various policies had empirical distributions of their discounted long-run costs generated through simulation. Among these policies, there was a policy believed to be better than the rest in the sense that the theoretical expected discounted long-run costs starting from each state was less or equal to that under any of the other policy. However, obtaining this theoretically optimal policy was computationally intensive as compared to the other policies of which some were heuristics that required almost negligible computation. For a system that has parameters that vary at various points throughout the day, the selected policy would have to be updated. Hence, if the theoretically optimal policy were to be chosen it would have to provide reduced system cost that justified its hefty computational budget.

This is not a simple trade-off to evaluate. However, establishing whether the simulated performances of the two policies are the same or not is a good starting point. In other words, the \emph{null hypothesis} that the means of the empirical distributions are equal is to be tested against an \emph{alternative hypothesis} that the mean of the optimal policy is the smaller of the two.

One approach to assessing the simulated performances would be to perform a significance test between the optimal discounted policy $\pi^*$ and its alternative $\pi$ at each state $x \in \mathcal{X}$. Put differently, a significance test take place between the \emph{value functions} \cite{suttonRLbook} $J_{\pi}^{\alpha}(x)$ such that the null hypothesis $J_{\pi^*}^{\alpha}(x) - J_{\pi}^{\alpha}(x) = 0$ would evaluated to see if it is \emph{rejected} in favour of $J_{\pi^*}^{\alpha}(x) - J_{\pi}^{\alpha}(x) < 0$.

Various issues plague such as approach. Firstly, a total of $|\mathcal{X}|$ statistical tests would have to be performed. Secondly, it is unclear as how to proceed if not all null hypotheses are either rejected or fail to be rejected. Thirdly, each state $x_k$ would require a sufficient amount of sampled trajectories to be simulated using $x$ as the initial state $x_0$. Lastly, such a test would clarify whether \emph{state-based performance} was better but not \emph{system-based performance}. Arguably, the latter performance measure is what the entity responsible for the system would be interested in optimising.

As to satisfy these four issues, this paper proposes two scalar metrics that represent system-based performance of discounted policies such that they can be compared via a single significance test. Two additional metrics are given for the average-cost system-based performance as well. Furthermore, connections between the discounted and average-cost metrics are established through the theory of \emph{Blackwell optimality}. These metrics would not require a minimum amount of trajectories to be sampled from each state but would require a minimum amount of trajectories to be sampled starting from states that were drawn from a probability distribution $f:\mathcal{X} \to \mathbb{R}_{\geq 0}$ where $\int_{\mathcal{X}} f(x)\, dx = 1$ and $x_0 \sim f$. Naturally, these scalar metrics are influenced by the choice of $f$. This paper proposes the use of the \emph{uniform distribution} as well as the \emph{stationary distribution} of the system's ergodic Markov chain under its policy as denoted by $\phi_{\pi}$. The reasoning, consequences and means of interpreting the scalar metrics under these two choices are addressed.

The developed theory is first applied to a simple problem where \emph{Markov Decisions Processes} (MDPs) are randomly generated as to satisfy the assumptions made by the theory. Such an approach is also useful in assessing a number of MDPs under different sizes and parameters. The theory is then applied to a queue control problem which takes the form of a \emph{Continuous-time Markov Decision Process} (CTMDP). Using \emph{uniformisation}, the CTMDP is converted to a MDP under which the theory can be applied. The results show that theoretically optimal policies are not always optimal in a statistically significant sense.

\newpage

\section{Background and literature}

\subsection{Markov Decision Processes}

Markov Decision Processes are used as a framework for sequential decision-making under uncertainty\footnote{The uncertainty refers to stochastic dynamics of the system and not model uncertainty. Model uncertainty requires \emph{robust optimisation}.} when a model describing the stochastic dynamics of the be environment is available. One of the main characteristics that distinguishes different MDPs is the specification of how long it will operate. \emph{Finite-horizon} MDPs have a known termination time or goal-state at which the system ceases to operate. This allows the \emph{total cost} of the system to be bounded. As a result, the total-cost criteria is an appealing and intuitive objective function that can be used to search for an optimal policy in a finite-horizon MDP. When the operating duration is infinite, irrelevant, unbounded or unknown then an \emph{infinite-horizon} MDP is necessitated. The total-cost has a meaningless infinite value which cannot be used to search for an optimal policy. Two methods of bounding the objective function is to either introduce a \emph{discount factor} as in section~\ref{section: discounted cost} or to compute the \emph{average-cost} with respect to some time-step as in section~\ref{section: average cost}. 

This paper focuses in the infinite-horizon MDP with \emph{finite state-space and actions}\footnote{This is an important specification as it guarantees a finite number of policies to exist from which one is optimal.} which can be described through the following tuple.

\begin{definition}[\textbf{MDP}]
\begin{equation}\label{eq:mdp tuple}
\mathcal{M} = \left(  \mathcal{X},\mathscr{A},\mathcal{P},\mathscr{C} \right)  
\end{equation}
where
\begin{itemize}
    \item[$\mathcal{X}$\quad] is the state space. In \emph{finite} $n$-dimensional MDPs this consists of a bounded set of integers $\mathcal{X} = \bigtimes_{i=1}^n \mathbb{Z}\cap [\underline{m}_i,\overline{m}_i]$ where $\bigtimes$ denotes the Cartesian product and $\infty<\underline{m}_i<\overline{m}_i<\infty$. If $\mathcal{X}$ has a \emph{continuous} component then an \emph{infinite} dimensional MDP arises.
    \item[$\mathscr{A}$\quad] is the action space. At decision epochs, this space is queried with regards to finite set of available actions $\mathscr{A}\left(x_n\right)$ where $\mathscr{A}: \mathcal{X} \to \mathbb{Z}\cap [\underline{m},\overline{m}]$ and $\infty<\underline{m}<\overline{m}<\infty$.
    \item[$\mathcal{P}$\quad] is the transition model that exhibits the \emph{Markov property}\footnote{The future given the present is independent of the past such that the present state is a sufficient statistics in predicting the probability of transitions to future states.}. If $\mathbf{P}^a$ is a $|\mathcal{X}| \times |\mathcal{X}|$ Markov chain that describes the probability of transitioning from state $x_k$ to $x_{k+1}$ under the execution of action $a_k=a$ then $\mathcal{P}$ defined as the set $\mathcal{P}=\left\{  \mathbf{P}^a : a \in \mathscr{A} \right\}$. Note that each entry of $\mathbf{P}^a$ describes the probability $\mathscr{P}(x_{k+1}=j|x_k=i,a_k=a)$ where $i$ and $j$ refer to the row and column index, respectively. If $N=|\mathcal{X}|$, each row $i$ requires its first ($N-1$) entries to lie \emph{within} a $(N-1)$-dimensional unit simplex such that $\sum_{j=1}^{N-1} \mathbf{P}_{ij}^a \leq 1$ and $\mathbf{P}_{ij}^a \geq 0$ where the last entry is free to vary accordingly $\mathbf{P}_{iN}^a  = 1 - \sum_{j=1}^{N-1} \mathbf{P}_{ij}^a$. This is the same as constraining it to lie \emph{on} a $N$-dimensional unit simplex $\mathbf{\Delta}^{(N)} = \left\{ (x_1,\cdots,x_N): \sum_{j=1}^N x_j = 1, x_j \geq 0 \, \forall j \in \mathbb{Z}\cap [1,N]  \right\}$. This paper will use the convention $\mathbf{P}_i^a \in \mathbf{\Delta}^{(N-1)}$ to denote the fact that the first $(N-1)$ independent elements lie within the unit simplex while the last dependent term satisfies the constraint that all entries sum to one.
    \item[$\mathcal{C}$\quad] is the cost model is is defined to be the set $\mathcal{C} = \left\{ \vec{C}^a: a \in \mathscr{A} \right\}$ where $\vec{C}^a$ are \emph{row-vectors} of length $|\mathcal{X}|$ that describe the \emph{lump-sum cost} incurred upon entering a state. The lump-sum cost at state $x_k$ consists of a \emph{holding cost} $C_H: \mathcal{X} \times \mathscr{A} \to \mathbb{R}_{\geq 0}$ and \emph{transition cost} $C_T: \mathcal{X} \times \mathcal{X} \times \mathscr{A} \to \mathbb{R}_{\geq 0}$ such that $\vec{C}^a_i = C(i,a) = C_H(i,a) + \sum_{j\in \mathcal{X}} \mathbf{P}_{i,j}^a C_T(i,j,a)$. It is often the case that the holding cost does not depend on the action taken $C_H: \mathcal{X} \to \mathbb{R}_{\geq 0}$ and is thus independent of the policy.
\end{itemize}
\end{definition}

A MDP is controlled by a policy $\pi$ that through querying it for an action $a_k$ after a transition from $x_{k-1}$ to $x_{k}$ has completed. The most general policy requires the entire history of the system up to state $x_k$ as an argument an returns a probability distribution over $\mathscr{A}$ where infeasible actions are allocated a zero probability. Such a history-dependent stochastic policy is \emph{intractable} to compute. Fortunately, it has been shown that there exists a stationary, deterministic and Markov policy $\pi_{MD}: \mathcal{X} \to \mathscr{A}$ that performs at least as well as the best history-dependent stochastic policy (see proposition 1.1.1 of\cite{BertsekasVol2}). Hence, without any loss of performance this paper restricts itself to such a class of policies where a $|\mathcal{X}|$-length \emph{column} vector $\vec{\pi}$ will be defined to have each entry given by $\pi_{MD}(x)$. In the sequel, all vectors are assumed to be column vectors and $\pi$ will be assumed to denote $\pi_{MD}$.\\

A policy serves as instructions to build the model $\left(\mathbf{P}^{\pi},\vec{C}^{\pi}\right)$ from $\mathcal{P}$ and $\mathcal{C}$. In other words, the transition model is constructed row-wise as $\mathbf{P}_i^{\pi} = \mathbf{P}_{i}^{a=\pi(i)}$ while the cost model is constructed element-wise $\vec{C}_i^{\pi} = \vec{C}_i^{a=\pi(i)}$. This model is then used to \emph{evaluate} one of the state-based performance metrics ($J$ and/or $h$) as discussed in the next two sections. A policy $\pi^* \in \Pi = \mathscr{A}^{|\mathcal{X}|}$ is deemed optimal if $\forall \pi \in \Pi, \forall x \in \mathcal{X}: \, J_{\pi^*}(x) \leq J_{\pi}(x)$. Such an optimal stationary, deterministic and Markov policy can be obtained using numerical iterative algorithms such as \emph{Value Iteration} or \emph{Policy Iteration} \cite{BertsekasVol2,gosavi2015simulation}. Alternatively, the MDP can be cast as a \emph{Linear Program} \cite{puterman2014markov,sheskin_markov} for which several solution methods exist such as the \emph{Simplex algorithm} \cite{hillier_lieberman_OR_book}.

\subsubsection{Infinite-horizon discounted cost}\label{section: discounted cost}

The discounted infinite horizon cost for a given state $x \in \mathcal{X}$ under policy $\pi$ is defined as
\begin{eqnarray}
    J_{\pi}^\alpha(x) & = &  \mathbbm{E}\left[ \sum_{k=0}^\infty C\left(x_k,\pi(x_k)\right)\,\bigg|\, x_0 = x \right]
\end{eqnarray} 
which can be written in matrix form
\begin{eqnarray}
\vec{J}_{\pi}^\alpha & = & \sum_{k=0}^\infty \alpha^k \mathbf{P}_{\pi}^k \vec{C}_{\pi}\\
& = & \left(  \sum_{k=0}^\infty \alpha^k \mathbf{P}_{\pi}^k  \right) \vec{C}_{\pi} \\
& = & \left( \mathbf{I} - \alpha \mathbf{P}_{\pi} \right)^{-1} \vec{C}_{\pi} \label{eq: discounted bellman policy evaluation}
\end{eqnarray}
Equation~(\ref{eq: discounted bellman policy evaluation}) is referred to as the \emph{discounted Bellman policy evaluation equations}. This set of linear equations has a \emph{unique solution} and the matrix inverse always exists for $\alpha < 1$ as the largest eigenvalue of $\mathbf{P}_{\pi}$ always lies in the unit circle $\rho_1(\alpha \mathbf{P}_{\pi}) = \alpha \rho_1(\mathbf{P}_{\pi}) = \alpha < 1$ where $\rho_1(\cdot)$ denotes the largest/dominant eigenvalue.\\

\subsubsection{Infinite-horizon average cost}\label{section: average cost}

The infinite horizon average-cost state-value functions are similarly defined as
\begin{eqnarray}
J_{\pi}(x) = \lim_{N\to\infty} \frac{1}{N} \mathbbm{E}\left[ \sum_{k=0}^{N-1} C\left(x_k,\pi(x_K)\right) \,\bigg|\, x_0 = x \right]
\end{eqnarray}
where the matrix formulation follows as
\begin{eqnarray}
    \vec{J}_{\pi} & = & \lim_{N \to \infty}\frac{1}{N} \sum_{k=0}^{N-1} \mathbf{P}_{\pi}^k \vec{C}_{\pi} \\
    & = & \left(  \lim_{N \to \infty}\frac{1}{N} \sum_{k=0}^{N-1} \mathbf{P}_{\pi}^k \right)\vec{C}_{\pi} \\
    & = & \mathbf{P}_{\pi}^* \vec{C}_{\pi}. \label{eq: average costs state-value}
\end{eqnarray}

In equation~(\ref{eq: average costs state-value}), the matrix $P^*$ is called the \emph{Cesaro limit} and for ergodic chains $\mathbf{P}_{\pi}$ it is the \emph{limiting matrix} $\mathbf{P}_{\pi}^* = \lim_{N \to \infty}\frac{1}{N} \sum_{k=0}^{N-1} \mathbf{P}_{\pi}^k =\lim_{N \to \infty} \mathbf{P}_{\pi}^N  $ such that each row is the \emph{limiting distribution} $\vec{p}_{\infty}$ (see appendix A4 of \cite{puterman2014markov}). Moreover, if $\mathbf{P}$ is \emph{ergodic} (see chapter 7 of \cite{cassandras_book}) then $\vec{p}_{\infty}=\vec{\phi}_{\pi}$ which is the solution to $\vec{\phi}_{\pi} =  \mathbf{P}_{\pi}^T\vec{\phi}_{\pi}$ with the constraint that $ \mathbf{1}^T \vec{\phi}_{\pi} = 1$ such that $\vec{\phi}_{\pi}$ is the \emph{stationary distribution}. Note that $\mathbf{1}$ is a column vector of ones, $\mathbf{I}$ is the \emph{identity matrix} and $\vec{\phi} \in \mathbf{\Delta}^{(|\mathcal{X}|-1)}$ is a column vector. From this it can be seen that $\mathbf{P}_{\pi}^* = \mathbf{1} \,(\vec{\phi}_{\pi})^T$. A result that further pertains to a chain with a \emph{single recurrent class} is that $ \forall x,y \in \mathcal{X}: \, J(x) = J(y) = J$ such that $\vec{J} = J \mathbf{1}$. Hence, the following expressions holds for this \emph{uni-chain} case
\begin{eqnarray}
    J_{\pi} & = & \left(\vec{\phi}_{\pi}\right)^T \vec{C}_{\pi}  \label{eq: scalar J}\\
    \vec{J}_{\pi}& = &\mathbf{1} \left(\vec{\phi}_{\pi}\right)^T\vec{C}_{\pi} \label{eq: vector J} \\
    & = & \mathbf{P}_{\pi}^* \vec{C}_{\pi} \nonumber .
\end{eqnarray}

While equations (\ref{eq: vector J}) and (\ref{eq: average costs state-value}) seem useful and intuitive, they do not provide much insight in formulating Bellman equations or providing a \emph{Dynamic Programming} solution. However, the literature on Markov Decision Processes does provide Bellman Equations for the average cost case \cite{BertsekasVol2,gosavi2015simulation,puterman2014markov}. The Bellman equations are different from the discounted-case in that they do not have a unique solution. Instead an infinite family of solutions exist which differ by some additive constant. To further the discussion, these Bellman equations are provided below as for the general case (more than one recurrent class)
\begin{eqnarray}
    \vec{J}_{\pi} & = & \mathbf{P}_{\pi}  \vec{J}_{\pi} \label{eq: av bellman 1}\\
    \vec{J}_{\pi} + \vec{h}_{\pi}& = & \vec{C}_{\pi} + \mathbf{P}_{\pi}  \vec{h}_{\pi} \label{eq: av bellman 2}.
\end{eqnarray}
It immediately becomes clear that \emph{two} systems of equations are given than the usual single system as in the discounted case. In the average-cost literature, $\vec{J}_{\pi}$ is referred to as the \emph{gain} and $\vec{h}_{\mu}$ as the \emph{bias}. The reasoning behind these terms will hopefully become clear. The solution to equations~(\ref{eq: av bellman 1}) and (\ref{eq: av bellman 2}) is not unique. More specifically, the bias can take on any additive constant $d$ that satisfies $\vec{d} = \mathbf{P}_{\pi}\vec{d}$ such that $(\vec{J}_{\pi},\vec{h}_{\pi}+\vec{d})$ is a \emph{gain-optimal} solution (see proposition 5.1.9 of \cite{BertsekasVol2}).

In the special case where only one recurrent class exists, the solution can be simplified. A single set of Bellman equations can be solved for in evaluating a policy
\begin{equation}
    \vec{h}_{\pi} + J\mathbf{1} = \vec{C}_{\pi} + \mathbf{P}_{\pi}  \vec{h}_{\pi}  \label{eq: unichain average reward bellman}
\end{equation}
where there are $|\mathcal{X}|+1$ unknowns and $|\mathcal{X}|$ equations. This issue of an \emph{undetermined} system is resolved in various ways (see chapter 8 of \cite{puterman2014markov}, chapter 6 of \cite{gosavi2015simulation} and section 5 of \cite{BertsekasVol2}). The most common approach is to select a \emph{distinguished state} $x^\# \in \mathcal{X}$ such that $h_{\mu}(x^\#) = 0$ which results in \emph{relative Policy Iteration} and \emph{relative Value Iteration} algorithms.

An important caveat\label{text: gain optimal set} of the uni-chain average-cost Bellman policy equations is that while $(J,\vec{h})$ is a unique solution under $\pi$; it cannot be said that $J$ is unique to $\pi$. The uniqueness stems from $\vec{h}$ being unique (see proposition 7.4.1 of \cite{BertsekasVol1}). This can be heuristically explained through viewing (\ref{eq: av bellman 2}) or (\ref{eq: unichain average reward bellman}) as an analog of (\ref{eq: discounted bellman policy evaluation}) once the gain has been solved for. Hence, in solving for $\vec{h}$ as if it were $\vec{J}$ then the uniqueness result carries over from (\ref{eq: discounted bellman policy evaluation}). Various policies may have the same gain but a different an unique bias. Algorithms such as \emph{Policy Iteration} perform policy improvement in two parts \cite{average_reward_RL}. First, the gain is optimised until $J_{k+1} = J_{k}$ and $h_{k+1}(x)\leq h_k(x), \, \forall x \in \mathcal{X}$. This can be viewed as the algorithm tuning the policy as to find the optimal recurrent class with optimal actions. The second parts then improves the bias until $h_{k+1}(x)= h_k(x), \, \forall x \in \mathcal{X}$ after which the algorithm terminates. In other words, it is optimising the prescribed actions over the transient states. All policies in the second part $\pi \in \Pi_2$ are gain-optimal and share the same value $J^*$ such that $J^* < J_{\pi}, \, \forall \pi \in \Pi\setminus \Pi_2 $. As there are a finite number of policies due to the assumption that the state-space and action-set is finite \cite{BertsekasVol1,BertsekasVol2}, a single policy $\pi \in \Pi_2$ will have an optimal bias $\vec{h}^*$.

\subsection{Hypothesis tests}

To interpret/investigate data an \emph{assumption} is usually made and \emph{tested}. Such an assumption is referred to as the \emph{null hypothesis} $H_0$ and can be rejected in favour of an \emph{alternative hypothesis} $H_A$ (another logical assumption). A statistical test can only reject the null hypothesis --- it cannot be accepted nor is the alternative hypothesis accepted upon its rejection. Hence the only outcomes are \emph{rejection} or \emph{failure to reject}. A statistical test derives a statistic which is used to report a $p$-value. This $p$-value is often acquired through querying a dedicated table with the relevant statistic. A $p$-value is defined as the probability of obtaining a datum at least as extreme as that found in the observed sample set \cite{nonparametric_tests_step_by_step,rice_stats}. It is used in conjunction with a significance level $\zeta \in (0,1)$ which is usually set to $\zeta=0.05$. If $p\leq \zeta$ then a result is said to be \emph{significant} such that the null hypothesis is rejected. Otherwise, it is not significant and the test fails to reject the null hypothesis.

\subsubsection{Student's t-test}\label{section: t test}
The \emph{one-sample} Student's $t$-test is a \emph{parametric} statistical test that assesses the null hypothesis that the \emph{sample mean} $\bar{X}$ is equal to some hypothesized \emph{population mean} $\mu$ \cite{rice_stats}. Mathematically, this can be expressed as $H_0: \bar{X} = \mu$. The $t$-statistic is computed as
\begin{equation}
    t = \frac{\left(\bar{X}-\mu\right)\sqrt{N}}{\hat{\sigma}}
\end{equation}
where $N = |X|$ is the sample size and $\hat{\sigma}$ is the sample standard deviation which uses $N-1$ as a denominator
\begin{equation}
    \hat{\sigma} = \left(\frac{\sum_{x\in X} (x-\bar{X})^2}{N-1}  \right)^{\frac{1}{2}}. \label{eq: sample std}
\end{equation}
As this is a parametric test, the $p$-value can be obtained from the Student's $t$-distribution $F_t$ directly where $F_t$ is a cumulative distribution function with $N-1$ degrees of freedom. If $|t|$ denotes the absolute value then a test against the \emph{two-sided} alternative hypothesis $H_A: \bar{X} \neq \mu$ has a $p$-value determined as $p_{\neq} = 2(1-F_t(|t|))$. The one-sided alternatives follows as $p_{<} = F_t(t)$ for $H_A: \bar{X} < \mu$ and $p_{>} = 1 - F_t(t)$ for $H_A: \bar{X} > \mu$. This test traditionally assumes that $X\sim \mathcal{N}$ is normally distributed \cite{rice_stats}. In practice, this assumption is often violated to some minor degree especially if the data-sets are considered large $N\geq100$ \cite{normality_assumption}.

\subsubsection{Welch's t-test}\label{section: welch}

The Welch's $t$-test is a parametric two-sample statistical test that assesses the null hypothesis of whether two samples have equal means $H_0: \bar{X} = \bar{Y}$ under the assumption that $\operatorname{var}(X)\neq \operatorname{var}(Y)$ \cite{welchs_t_test}. This assumption distinguishes it from the two-sample unpaired $t$-test. It requires a $t$-statistic
\begin{equation}
    t = \frac{\bar{X}-\bar{Y}}{\sqrt{s_X^2 + s_Y^2}}
\end{equation}
where $s^2$ is the \emph{sample standard-error} that can be obtained form the sample's standard deviation~(\ref{eq: sample std})
\begin{equation}
    s^2 = \frac{\sigma}{\sqrt{N}}.
\end{equation}
The $p$-values are obtained from a $t$-distribution $F_t$ with $\xi$ degrees of freedom. If $\xi_X = N_X - 1 = |X| - 1$ and $\xi_Y = N_Y - 1 = |Y| - 1$ then 
\begin{equation}
    \xi = \frac{\left( \frac{\sigma_X^2}{N_X} +\frac{\sigma_Y^2}{N_Y}  \right)^2}{\frac{\sigma_X^4}{N_X^2 \xi_X} +\frac{\sigma_Y^4}{N_Y^2 \xi_Y}}
\end{equation}
Using $F_t$ and $t$, one can obtain $p_{\neq}$, $p_{<}$ and $p_{>}$ as in the previous section. This test assumes that both $X$ and $Y$ are normally distributed and that in $X$ and $Y$ are unpaired. Despite the normality assumption, it has been reported to be robust against skewed distributions if the sample sizes are large \cite{robust_t}.

\subsubsection{Mann-Whitney U-test}\label{section:mann-whitney}
The Mann-Whitney $U$-test is a \emph{non-parametric}\footnote{This means that it sis not based in a parameterised probability distribution. Hence, tables play an important role in obtaining the $p$-values.} statistical test that assesses the null hypothesis that two samples $X\sim F_X$ and $Y\sim F^Y$ have the \emph{same distribution} where $F$ is a cumulative distribution function \cite{nonparametric_tests_step_by_step,rice_stats}. This can be mathematically expressed as $H_0: P(x\in X > y \in Y) = 1/2 =P(x\in X < y \in Y)$. This is the same as claiming the two distributions to be \emph{stochastically equal} $X \stackrel{s.t.}{=} Y$ such that $\forall z \in Z: F_X(z) = F_Y(z)$ where $Z$ is the support $X\cup Y \subseteq Z$. The $U$-statistic is obtained as
\begin{equation}
    \mathcal{U}(X,Y) = \sum_{x \in X}\sum_{y\in Y} S(x,y)
\end{equation}
where
\begin{equation}
    S(x,y) = \begin{cases}
    1, & x > y\\
    \frac{1}{2} , & x = y \\
    0, & x < y
    \end{cases}
\end{equation}
such that the final statistic is $U = \min\left\{\mathcal{U}(X,Y),\mathcal{U}(Y,X)\right\}$.

\subsubsection{Normality test}\label{section:normality test}
This paper uses the \emph{D'Agostino's} $k^2$ \emph{test} to assess the null hypothesis that a sample is normally distributed $H_0: X \sim \mathcal{N}$ \cite{norm_test}. The statistic of interest is $k^2 \sim \chi^2(\xi=2)$ such that it is a parametric test. This statistic is based on transforms of the \emph{skewness} $g_1$ and \emph{kurtosis} $g_2$ of the sample such that 
\begin{equation}
    k^2 = Z_1(g_1)^2 + Z_2(g_2)^2.
\end{equation}
If the $i^{th}$ \emph{central sample moment} is defined to be 
\begin{equation}
    m_i =\frac{1}{|X|} \sum_{x\in X}(x-\bar{X})^i
\end{equation}
then the skewness is determined as
\begin{equation}
    g_1 = \frac{m_3}{\left(m_2\right)^{\frac{3}{2}}}
\end{equation}
while kurtosis similarly follows
\begin{equation}
    g_2 = \frac{m_4}{\left(m_2\right)^2}.
\end{equation}
Sample skewness and kurtosis have distributions which are asymptotically normal such that central moments for these distributions $m_i(g_k)$ were derived in \cite{pearson1931_normality_test}. These expressions are all defined in terms of the sample size $N=|X|$. Although as asymptotically normal, transformations $Z_1$ and $Z_2$ are used to make the distributions as close to a standard normal as possible. Different transforms exist but the most widely used are
\begin{equation}
    Z_1(g_1) = \frac{1}{\ln{(W)}}\sinh^{-1}{\left(g_1\sqrt{\frac{W^2-1}{2 \,m_2(g_1)}}\right)}
\end{equation}
where
\begin{eqnarray}
     m_2(g_1) & = & \frac{6(N-2)}{(N+1)(N+2)}\\
     W & = & \sqrt{2\,G_2(g_1)+4} -1\\
     G_2(g_1) & = & \frac{36(N-7)(N^2+2N-5)}{(N-2)(N+5)(N+7)(N+9)}
\end{eqnarray}
such that $m_2(g_1)$ and $G_2(g_1)$ are the variance\footnote{This is because the mean is zero $m_1(g_1)=0$.} and kurtosis of $g_1$, respectively. The second transform follows as
\begin{eqnarray}
    Z_2(g_2) = \sqrt{\frac{9A}{2}}\left( 1 - \frac{2}{9A}  - \sqrt[3]{\frac{1 -\frac{2}{A}}{1 + \frac{g_2 - m_1(g_2)}{\sqrt{m_2(g_2)}}\times \sqrt{\frac{2}{A-4}}}   }  \right)
\end{eqnarray}
where 
\begin{eqnarray}
    m_1(g_2) & = & -\frac{6}{N+1} \\
    m_2(g_2) & = & \frac{24N(N-2)(N-3)}{(N+1)^2(N+3)(N+5)}\\
    A & = & 6 + \frac{8}{G_1(g_2)}\left(\frac{2}{G_1(g_2)} + \sqrt{1 + \frac{4}{(G_1(g_2))^2}}   \right)\\
    G_1(g_2)  & = &  \frac{6(N^2-5N+2)}{(N+7)(N+9)}\sqrt{ \frac{6(N+3)(N+5)}{N(N-2)(N-3)}  }
\end{eqnarray}
such that $G_1(g_2)$ is the skewness of $g_2$. If $F$ is the cumulative distribution function of $\chi^2(\xi=2)$ then $p=1 -F(k^2)$.

\subsection{Literature review}

This paper was motivated by the fact that a system-based performance metric or means of comparing MDP policies through a scalar could not be found. Hence, there is a lack of similar literature. However, assessing the significance of optimality through hypothesis tests is not new and is abundant in \emph{statistical decision theory} (see part two of \cite{parmigiani_decision_theory}). Furthermore, the use of distributions in MDPs is not novel either. Distributions over state value functions of MDPs in \emph{Reinforcement Learning} have shown promising results \cite{bellemare_distributional_MDP}. Such distributions can be \emph{parametric} \cite{morimura_parametric} or \emph{non-parametric} \cite{morimura_nonparametric}. The latter has been used in \emph{empirical Dynamic Programming} \cite{haskell_empirical_DP}. Distributions allow for a host of additional statistics other than the mean to be evaluated and used in evaluating an objective function. The inclusion of variance and risk were the key motivations behind \cite{morimura_parametric,morimura_nonparametric}.

\newpage

\section{Scalar MDP performance metrics}

This section questions whether the usual state-based optimality criteria used in algorithms that solve for the optimal infinite-horizon policies imply the system to be optimal in an overall sense as measured by some scalar.

\subsection{State-based and system-based optimality}

The goal of any Markov Decision Process (MDP) algorithm is to find the optimal policy $\pi^* \in \Pi$ where $\Pi = |\mathscr{A}|^{|\mathcal{X}|}$ is the finite policy space and $\mathcal{A}$ is the set of actions.  A MDP policy $\pi^* \in \Pi$ is deemed optimal if for any generic\footnote{Generic refers to this including not only both the discounted and average cost infinite horizon problems but also finite horizon problems} $\vec{J} \in \mathcal{J} = \mathbb{R}^{|\mathcal{X}|}$ if the following \emph{partial ordering} is satisfied
\begin{eqnarray}
    \forall \pi \in \Pi\setminus\{\pi^*\}, \forall x \in \mathcal{X}: & & J_{\pi^*}(x) \leq J_{\pi}(x)\label{eq: state-based optimality} \\
    \forall \pi \in \Pi\setminus\{\pi^*\}: & & \vec{J}_{\pi^*} \preceq  \vec{J}_{\pi}
\end{eqnarray}
where at least one \emph{strict inequality} $<$ holds \cite{BertsekasVol2,suttonRLbook,gosavi2015simulation}. As $\Pi$ is finite, concern need not be given to the case where $\forall \pi,\pi' \in \Pi_A: \,\vec{J}_{\pi'} \preceq  \vec{J}_{\pi} $ and $\forall \pi,\pi' \in \Pi_B: \,\vec{J}_{\pi'} \succ  \vec{J}_{\pi} $ such that $\Pi_A \cap \Pi_B = \varnothing$ and $\Pi_A \cup \Pi_B = \Pi$ (see chapter 5.6 of \cite{BertsekasVol2}).  A policy that is \emph{stationary, deterministic} and \emph{history-independent} (Markov)  $\pi^*: \mathcal{X} \to \mathscr{A} $  will allocate a single optimal action for each state. Hence, the system acts optimally in all states given its performance metric/objective function. Intuition would suggest that acting optimally in each state should yield an optimal system in the long-run. More specifically, for the infinite horizon problems, a stationary system that satisfies the \emph{state-based optimality} (\ref{eq: state-based optimality}) should satisfy some overall optimality criterion. Moreover, if the system has a stationary distribution $\vec{\phi} \in \Phi = \mathbf{\Delta}^{|\mathcal{X}|-1}$ then it will visit each state $x$ according to this distribution where it will execute an action $\pi(x)$, incur a one-step cost $C(x)$ and be allocated a state-value function $J(x)$ (its expected long-run performance). Such reasoning suggests that a stationary system should have the following \emph{system-based optimality} criterion
\begin{eqnarray}
    \eta_{\pi} = \left(\vec{\phi}_{\pi}\right)^T \vec{J}_{\pi} \label{eq: system-based optimality}
\end{eqnarray}
where $\eta: \Phi \times \mathcal{J} \to \mathbb{R}$. To be precise, (\ref{eq: system-based optimality}) will be referred to as \emph{stationary system-based optimality}. This will avoid confusion with its proposed alternative, \emph{uniform system-based optimality}, as defined below
\begin{equation}
\nu_{\pi} = \vec{U}^T \vec{J}_{\pi} \label{eq: uniform system-based optimality}
\end{equation}
where $\vec{U} = (|\mathcal{X}|)^{-1}\mathbf{1}$ is a vector of uniform probabilities. The stationary metric assumes the policy to have been in operation long enough to have induced the ergodic system to be in its stationary regime. Hence, the performance is only observed once transience has been removed. In contrast, the uniform metric attempts to account for transience and assumes a system is observed with no prior knowledge of what state it might be in. The policy is then executed which eventually drives the system into its stationary regime. 

It can be said that (\ref{eq: system-based optimality}) is suitable for gauging the average system-based performance while (\ref{eq: uniform system-based optimality}) accounts for the transient operation of the system under the fixed policy until the stationary regime is entered. This suggests a weighted metric
\begin{equation}\label{eq: hybrid metric}
\xi_{\pi}  = \theta \eta_{\pi} + (1-\theta) \nu_{\pi}
\end{equation}
where $\theta \in [0,1]$ suggests an importance of long-run performance over transient performance. For a system that takes relatively long to exit transience, $\theta$ should be reduced. Due to the setting of $\theta$ requiring research of its own, the use of this hybrid metric is outside the scope of this paper. Further work in this metric is suggested as it is most likely the most appropriate of the metrics.

It would be hoped for that state-based optimality would imply system-based optimality. Moreover, it is the latter that those responsible for the system would pay for in the long-run. The next two sections investigates this.

\begin{remark}
State-based optimality (\ref{eq: state-based optimality}) along with \emph{state-independence} forms the foundation of the \emph{Policy Improvement Theorem} \cite{suttonRLbook}. Briefly, if a policy can be changed to $\pi'$ at state $i$ by greedily selecting $\pi'(i) = a \in \mathscr{A}$ such that $J_{\pi'}(i) \leq J_{\pi}(i)$ without affecting other states $J_{\pi'}(k) = J_{\pi}(k), \, \forall k \in \mathcal{X}\setminus \{i\}$ then the theorem holds. This theorem is essentially the policy improvement step of Policy Iteration where all states undergo improvement as described above. The theorem is also found in Value Iteration but only the newly evaluated state has improvement performed on it.  
\end{remark}

\subsection{Average-cost optimality}\label{section: average-cost optimality}

This paper restricts itself to MDPs that have policies $\pi \in \Pi$ that result in $\mathbf{P}_{\pi}$ being uni-chain. Doing so has the following outcomes:
\begin{itemize}
	\item The gain is state-independent such that $\vec{J} = J\mathbf{1}$ as in equation~(\ref{eq: unichain average reward bellman}). The intuition behind this follows from \cite{average_reward_RL}. The recurrent states are visited infinitely often and as such the expected cost across these states cannot differ. Meanwhile, the finite cost incurred during passage between the transient states become negligible in the limit of the infinite horizon problem. 
	\item If the uni-chain consists of a single recurrent class then it is recurrent chain\footnote{Additional assumptions such as aperiodicity ensures it to be ergodic \cite{stewart2009probability,cassandras_book}}. This means that no transient states exist such that all states of the MDP achieve the same gain. Recalling the discussion of average-cost Policy Iteration from page~\pageref{text: gain optimal set}, no second stage occurs as there are not transient states to optimise the bias over. Hence, $|\Pi_2|=1 $ such that there is a unique gain-optimal policy\label{text: unique gain optimal policy}.
\end{itemize}

Ensuring\label{text: enusre unichain} the uni-chain condition in practice is not straightforward. In fact, it has been shown in \cite{tsitsiklis_2007_unichain} that finding all policies that construct a uni-chain $\textbf{P}_{\pi}$ from $\mathcal{P}$ is an NP-hard problem. However, the \emph{weak accessibility} condition (see definition 5.2.2 of \cite{BertsekasVol2}) can be verified in polynomial time as to confirm a single state-independent gain $J$. It should be noted that in simple problem where $\mathbf{P}_a \in \mathcal{P}$ are all uni-chain and have the same structure (i.e. same recurrent class and transient states) such that they only differ in the values of their non-zero entries then the uni-chin assumption holds. This can be further extended to the case where $\mathbf{P}_a \in \mathcal{P}$ are all recurrent in which case $|\Pi_2|=1$ and a unique gain-optimal policy is guaranteed. While the gain is generally of greater interest than the bias, $\vec{h}$ is an important concept to discuss and optimise when possible. Doing so leads to better transient behaviour \cite{BertsekasVol2}.

The bias at some state $x$ is defined as
\begin{equation}
h_{\pi}(x) =  \lim_{N \to \infty} \mathbbm{E}\left[\sum_{k=0}^{N-1} \left( C\left(x_k,\pi(x_k)\right) - J_{\pi} \right) \big| x_0 = x \right]  \label{eq: bias}
\end{equation}
which can be written in matrix notation as
\begin{eqnarray}
\vec{h}_{\pi} & = & \sum_{k=0}^{\infty} \left( \mathbf{P}_{\pi}^k \vec{C}_{\pi} - J \mathbf{1}  \right) \label{eq: bias matrix} \\
& = & \vec{C}_{\pi} - J \mathbf{1} + \sum_{k=1}^{\infty} \left( \mathbf{P}_{\pi}^k  - \mathbf{P}^*_\pi  \right) \vec{C}_{\pi} \nonumber \\
& = &  - J \mathbf{1} +\vec{C}_{\pi} +  \sum_{k=1}^{\infty} \left( \mathbf{P}_{\pi}  - \mathbf{P}^*_\pi  \right)^k \vec{C}_{\pi} \label{step: pstar} \\
& = & - J \mathbf{1} +\sum_{k=0}^{\infty} \left( \mathbf{P}_{\pi}  - \mathbf{P}^*_\pi  \right)^k \vec{C}_{\pi} \nonumber\\
& = & \left(  \mathbf{I} - \mathbf{P}_{\pi} + \mathbf{P}^*_\pi  \right)^{-1} \vec{C}_{\pi} -  J \mathbf{1}\nonumber
\end{eqnarray}
where (\ref{step: pstar}) made use of the fact that for $k \geq 1$ it holds that $\mathbf{P}^k - \mathbf{P}^* = (\mathbf{P} - \mathbf{P}^*)^k$. The intuitive definitions of bias in (\ref{eq: bias}) and (\ref{eq: bias matrix}) were not formulated as a matter of convenience but can be derived from the \emph{Laurent series expansion} of the discounted-cost Bellman Equations (\ref{eq: laurent series 1}). Algorithms such Policy Iteration optimises $\vec{h}_{\pi_2}$ among the candidates $\pi_2 \in \Pi_2$ such that a \emph{bias-optimal} policy $\pi^*$ is found: $\vec{h}_{\pi^*} \preceq \vec{h}_{\pi_2}  $ where at least one strict equality must hold in the partial ordering. There exist a single unique bias-optimal policy (as discussed on page~\pageref{text: gain optimal set}) which is naturally a gain-optimal policy. Only if the uni-chain is a recurrent chain does uniqueness hold for both the terms gain-optimal and bias-optimal. State-based gain optimality can now be related to stationary system-based optimality.

\begin{theorem}\label{theorem: average costs implies stationary system optimality}
For the infinite horizon average-cost MDP that has an optimal policy $\pi^* \in \Pi_2 \subset \Pi$ which induces a uni-chain $\mathbf{P}_{\pi^*}$, state-based optimality (gain-optimal) implies stationary system-based optimality 
\begin{equation}
    \forall \pi \in \Pi\setminus\{\pi^*\}: \quad \vec{J}_{\pi^*} \preceq  \vec{J}_{\pi} \implies \forall \pi \in \Pi\setminus\{\pi^*\}: \quad \eta_{\pi^*} \leq \eta_{\pi} \label{eq: average costs system optimality}
\end{equation}
where $\eta_\pi = J_{\pi}$ and $\eta_{\pi^*} = J_{\pi^*}$.
\end{theorem}
\begin{proof}
Using the definition of average-cost (\ref{eq: average costs state-value}) and multiplying through by $\vec{\phi}_{\pi}$ as to compute (\ref{eq: system-based optimality})
\begin{eqnarray}
    \eta_{\pi} & = & \left(\vec{\phi}_{\pi}\right)^T \mathbf{P}_{\pi}^* \vec{C}_{\pi} \nonumber\\
    & = & \left(\vec{\phi}_{\pi}\right)^T \mathbf{1} \left(\vec{\phi}_{\pi}\right)^T \vec{C}_{\pi}\nonumber\\
    & = & \left(\vec{\phi}_{\pi}\right)^T \vec{C}_{\pi} \nonumber\\
    & = & J_{\pi} \label{eq: eta equal J}
 \end{eqnarray}
 From the uni-chain condition it is noted that $\forall x \in \mathcal{X}: J_{\pi}(x) = J_{\pi}$ which makes state-based optimality independent of state. It follows that $|\Pi_2| \geq 1$ such that $J_{\pi^*}$ is not guaranteed to be unique to one policy. To summarise,
 \begin{eqnarray}
     \forall \pi \in \Pi\setminus\{\pi^*\}: & &  J_{\pi^*} \leq J_{\pi} \label{eq: optimal J}\\
     \therefore \forall \pi \in \Pi\setminus\{\pi^*\}: & &  \eta_{\pi^*} \leq \eta_{\pi} \label{eq: optimal eta}
 \end{eqnarray}
which completes the proof.
\end{proof}

\begin{corollary}
If $\mathbf{P}_{\pi^*}$ consists of a single recurrent class with no transient states then $|\Pi_2|=1$ such that $J_{\pi^*}$ is unique to a single policy. Hence, \emph{strict inequality} holds
\begin{eqnarray}
     \forall \pi \in \Pi\setminus\{\pi^*\}: & &  J_{\pi^*} < J_{\pi} \label{eq: optimal J strict}\\
     \therefore \forall \pi \in \Pi\setminus\{\pi^*\}: & &  \eta_{\pi^*} < \eta_{\pi} \label{eq: optimal eta strict}
 \end{eqnarray}
 such that strict stationary system-based optimality is achieved
 \begin{equation}
    \forall \pi \in \Pi\setminus\{\pi^*\}: \quad \vec{J}_{\pi^*} \preceq  \vec{J}_{\pi} \implies \forall \pi \in \Pi\setminus\{\pi^*\}: \quad \eta_{\pi^*} < \eta_{\pi} \label{eq: average costs system optimality strict}
\end{equation}
\end{corollary}

\begin{corollary}
For the uni-chains that consists of a single recurrent class $\forall \pi \in \Pi\setminus\{\pi^*\}:  \eta_{\pi^*} < \eta_\pi \implies \forall \pi \in \Pi\setminus\{\pi^*\} \vec{h}_{\pi^*} \preceq \vec{h}_\pi$.
\end{corollary}

\begin{corollary}
Equations (\ref{eq: optimal J}) --- (\ref{eq: optimal eta strict}) along with (\ref{eq: system-based optimality}) further suggest that system-based optimality is completely determined by $\vec{\phi}_{\pi}$ when $\forall \pi \in \Pi: \vec{C}_{\pi} = \vec{C}$ which is the case when only deterministic or stochastic \emph{state costs} are incurred that are policy independent. In other words, no \emph{transition costs} may be incurred. Hence, a policy tunes $\vec{\phi}_{\pi^*}$ as to allocate larger probabilities to low cost states.
\end{corollary}

\begin{corollary}
In a system with both state and transition costs, an optimal policy tunes both $\vec{\phi}_{\pi^*}$ and $\vec{C}_{\pi^*}$ as to minimise their weighted sum. In other words, the policy achieves its objective by fulfilling the trade-off between spending larger proportions of time in low costs states and traversing to and between low-costs states in inexpensive manner.
\end{corollary}

The uniform system-based criteria follows

\begin{theorem}
For the infinite horizon average-cost MDP that has an optimal policy $\pi^* \in \Pi_2 \subset \Pi$ which induces a uni-chain $\mathbf{P}_{\pi^*}$, state-based optimality (gain-optimal) implies uniform system-based optimality 
\begin{equation}
    \forall \pi \in \Pi\setminus\{\pi^*\}: \quad \vec{J}_{\pi^*} \preceq  \vec{J}_{\pi} \implies \forall \pi \in \Pi\setminus\{\pi^*\}: \quad \nu_{\pi^*} \leq \nu_{\pi} \label{eq: average costs system optimality uniform}
\end{equation}
where $\nu_\pi = J_{\pi}$ and $\nu_{\pi^*} = J_{\pi^*}$.
\end{theorem}
\begin{proof}
Using the fact that $\forall x \in \mathcal{X}: J(x) = J$ it follows that
\begin{eqnarray}
\nu_{\pi} & = & \vec{U} (J \mathbf{1})^T  \nonumber\\
& = & \sum_{x \in \mathcal{X}} \frac{J}{|\mathcal{X}|}\nonumber\\
& = & J \label{eq: nu equal J}
\end{eqnarray}
where $\forall \pi,\pi'\in \Pi_2: \, \nu_{\pi'} = \nu_{\pi} = J$ such that \emph{equality} holds but $\forall \pi^* \in \Pi_2, \forall \pi \in \Pi\setminus \Pi_2:\, \nu_{\pi^*} < \nu_\pi$ such that \emph{strict inequality} holds.
\end{proof}

\begin{corollary}
If $\mathbf{P}_{\pi^*}$ consists of a single recurrent class with no transient states then $|\Pi_2|=1$ such that $J_{\pi^*}$ is unique to a single policy and \emph{strict inequality} holds
\begin{equation}
    \forall \pi \in \Pi\setminus\{\pi^*\}: \quad \vec{J}_{\pi^*} \preceq  \vec{J}_{\pi} \implies \forall \pi \in \Pi\setminus\{\pi^*\}: \quad \nu_{\pi^*} < \nu_{\pi} \label{eq: average costs system optimality uniform strict}.
\end{equation}
\end{corollary}

It can be concluded that state-based optimality is a criteria that can be used to optimise both forms of system-based optimality. Furthermore, given that the $\mathbf{P}_{\pi^*}$ consists of a single recurrent class then $\nu_{\pi^*}$ and $\eta_{\pi^*}$ will be unique global optima. To the knowledge of this paper, there is no similar general condition to that of the weak-accessibility condition which would corroborate a single recurrent class to be induced by $\pi^*$ other than the practical suggestions from page~\pageref{text: enusre unichain}. In other words, if $\mathbf{P}_a \in \mathcal{P}$ all consist of a single recurrent class then the same holds for $\mathbf{P}_{\pi^*}$ such that strict equality holds. If one were to forgo the need for a prior guarantee of a single recurrent class then the MDP can be solved for and $\mathbf{P}_{\pi^*}$ can be investigated for having a single recurrent class. This can be achieved by verifying the limiting distribution to have no zero entries. Alternatively, graph-based algorithms can be used to detect \emph{communicating classes} such as \emph{Kosaraju’s algorithm} for finding strongly connected components (see listing 5-11 on page 111 of \cite{hetland2014python}).

\subsection{Discounted-cost optimality}\label{section: discounted-cost optimality}

The following theorem shows that uniform system-based optimality holds for a state-based optimal discounted-cost policy.

\begin{theorem}
For the infinite horizon discounted-cost MDP with no assumption of the chain structure of $\mathbf{P}_{\pi^*}$, state-based optimality implies strict uniform system-based optimality 
\begin{equation}
    \forall \pi \in \Pi\setminus\{\pi^*\}: \quad \vec{J}_{\pi^*}^\alpha \preceq  \vec{J}_{\pi}^\alpha  \implies \forall \pi \in \Pi\setminus\{\pi^*\}: \quad \nu_{\pi^*}^\alpha  < \nu_{\pi}^\alpha  \label{eq: discounted costs system optimality uniform}.
\end{equation}
\end{theorem}
\begin{proof}
Let $\epsilon(x) =J_{\pi}^\alpha- J_{\pi^*}^\alpha \geq 0 , \, \forall x \in \mathcal{X}$. If $\pi^*$ is optimal then for at least one $x^{\#} \in \mathcal{X}: \epsilon(x^{\#}) > 0$ given that $\pi^* \neq \pi$. Compute the uniform system-based performance for $\pi$
\begin{eqnarray}
\nu_{\pi}^\alpha & = & \vec{U}\left( \vec{J}_{\pi}^\alpha \right)^T \nonumber\\
& = & \frac{1}{|\mathcal{X}|} \sum_{x \in \mathcal{X}} J_{\pi}^\alpha(x) \nonumber\\
& = & \frac{1}{|\mathcal{X}|} \sum_{x \in \mathcal{X}} J_{\pi^*}^\alpha(x) + \frac{1}{|\mathcal{X}|} \sum_{x \in \mathcal{X}} \epsilon(x) \nonumber \\
& = & \nu_{\pi^*}^\alpha +  \bar{\epsilon} \nonumber
\end{eqnarray}
where $\bar{\epsilon} \geq \epsilon(x^{\#})/|\mathcal{X}|$ such that $\nu_{\pi}^\alpha > \nu_{\pi^*}^\alpha$.
\end{proof}

The proceeding proof illuminates the fact that uniform system-based optimality will hold for any state-based optimal policy. Furthermore, it is a unique global optima without any assumption made on the structure of $\mathbf{P}_{\pi^*}$. This is in contrast to that of the average-cost case. Hence, if a state-based optimal policy (optimality established as usual in the MDP literature) is applied to a system with no prior belief of the state-occupation distribution (such that only a uniform distribution appears reasonable) then the long-run system-based performance should be optimal.

Such positive results do not carry over to the discounted stationary system-based optimality. To see this in a less rigorous but intuitive manner, consider the fact that the the long-run discounted cost minimises transient behaviour\footnote{Setting $\alpha=0$ leads to \emph{greedy} or \emph{myopic} behaviour where the policy optimises only its immediate one-step cost. This is an example of extreme transient optimisation.} to some extent along with the following fictitious and biased example.
\begin{example}
The optimal policy has a state-value vector $\vec{J}_{\pi^*}^\alpha = [1,2,3]$ that induces a stationary distribution $\vec{\phi}_{\pi^*} = [0.1,0.1,0.8]$ such that $\eta_{\pi^*}^\alpha = 2.7$. Another policy has inferior state-based optimality $\vec{J}_{\pi}^\alpha = [1.1,2.1,3.1]$ but induces a stationary distribution $\vec{\phi}_{\pi} = [0.8,0.1,0.1]$ such that $\eta_{\pi}^\alpha = 1.4 < \eta_{\pi^*}^\alpha$ in which case its system-based performance is superior.
\end{example}
The obvious floor in this example is that the stationary distributions were subjectively tuned to produce a desired outcome. There is no result that suggests $\pi^*$ should have a worse stationary distribution than $\pi \in \Pi\setminus\{\pi^*\}$ other than the fact that a low enough $\alpha$ might favour optimising transient behaviour enough to do so. The next section derives the following relationship
\begin{equation}
    \eta_{\pi}^\alpha = \frac{\eta_{\pi}}{1-\alpha} \nonumber
\end{equation}
which connects discounted and average-cost stationary system-based optimality. This in conjunction with the existing theory of \emph{Blackwell optimality} and $n$-\emph{discount optimality} provides a vehicle to determine when state-based optimality implies stationary system-based optimality for the infinite horizon discounted case.

\subsection{The connection between discounted and average costs}

The following theorem is well known and can be found in chapter 5.1.1 of \cite{BertsekasVol2} or chapter 10.1.2 of \cite{puterman2014markov}.

\begin{theorem}
Under some policy $\pi \in \Pi$, the average and discounted infinite horizon costs are related as
\begin{equation}
    \vec{J}_{\pi} = \lim_{\alpha \to 1} \vec{J}^{\alpha}_{\pi}.
\end{equation}
\end{theorem}
Chapter 5.1.1 of \cite{BertsekasVol2} provides the following less rigorous proof along with a rigorous version. This paper makes use of the former and restates it below.
\begin{proof}
\begin{eqnarray}
\vec{J}_{\mu} & = & \lim_{N \to \infty} \frac{1}{N} \sum_{k=0}^{N-1} \mathbf{P}_{\mu}^k \vec{C}_{\pi} \nonumber \\
& = & \lim_{\alpha \to 1} \lim_{N \to \infty} \frac{\sum_{k=0}^{N-1} \alpha^k\mathbf{P}_{\mu}^k \vec{C}_{\pi}}{\sum_{k=0}^{N-1} \alpha^k} \nonumber \\
& = &\lim_{\alpha \to 1} \left\{ \frac{1}{ \lim_{N\to \infty} \sum_{k=0}^{N-1} \alpha^k} \times  \lim_{N\to \infty} \sum_{k=0}^{N-1} \alpha^k\mathbf{P}_{\mu}^k \vec{C}_{\pi} \right\} \nonumber \\
 & = & \lim_{\alpha \to 1} (1-\alpha)(\mathbf{I}-\alpha \mathbf{P}_\pi)^{-1} \vec{C}_{\pi}\nonumber\\
& = & \lim_{\alpha \to 1} \mathbf{M}_{\pi}(\alpha) \vec{C}_{\pi} \label{eq: M}\\
& = & \lim_{\alpha \to 1} (1-\alpha) \vec{J}_{\mu}^{\alpha} \nonumber
\end{eqnarray}
\end{proof}
In the proof of proposition 5.1.1 of \cite{BertsekasVol2}, it is noted that through the \emph{Cramer's rule}, the \emph{adjoint/adjugate matrix} and the subsequent expanding of the adjoint using \emph{Cayley-Hamilton theorem} then $\mathbf{M}_{\pi}(\alpha)$ can have its entries expressed as zeroes or fractions consisting of polynomial in $\alpha$. That is to say
\begin{equation}
    \left(\mathbf{M}_{\pi}(\alpha)\right)_{ij} = \frac{\gamma\prod_{k=1}^{|\mathcal{X}|}(\alpha-\zeta_k)}{\prod_{k=1}^{|\mathcal{X}|}(\alpha - \xi)}
\end{equation}
where $\gamma \in \mathbb{R}$, $\zeta_k \in \mathbb{R}$ and $\xi_k \in \mathbb{R}\setminus\{1\}$. Recalling that $\vec{J} = \mathbf{P}\vec{C}$, it can be seen from \ref{eq: M} that
\begin{equation}
    \mathbf{P}_{\pi}^* = \lim_{\alpha \to 1}\mathbf{M}_{\pi}(\alpha)
\end{equation}
such that a first-order \emph{Taylor-series expansion} of $\mathbf{M}_{\pi}(\alpha)$ about the neighbourhood of $\alpha=1$ leads to
\begin{equation}
    \mathbf{M}_{\pi}(\alpha) = \mathbf{P}_{\pi}^* + (1-\alpha)\mathbf{H}_{\pi} + \mathcal{O}\left((1-\alpha)^2\right) \label{eq: taylor M}
\end{equation}
such that $\lim_{\alpha \to 1} \mathcal{O}\left((1-\alpha)^2\right)/(1-\alpha)=0$. By the definition of a Taylor series, $(\mathbf{H})_{ij} = -\partial (\mathbf{M})_{ij}/\partial\alpha|_{\alpha=1}$. However, it can be shown that $\mathbf{H}$ is the \emph{Drazin inverse} $\mathbf{H} = (1-\mathbf{P})^{\#} = (\mathbf{I} -\mathbf{P} + \mathbf{P}^*)^{-1}(\mathbf{I} - \mathbf{P}^*) = (\mathbf{I} -\mathbf{P} + \mathbf{P}^*)^{-1} - \mathbf{P}^*$ where $(\mathbf{I} -\mathbf{P} + \mathbf{P}^*)^{-1}$ is the \emph{fundamental matrix} \cite{puterman2014markov}. Using this in (\ref{eq: M}) leads to the \emph{truncated Laurent Series expansion} of the discounted MDP
\begin{eqnarray}
    \vec{J}_{\pi}^\alpha & = &  (1-\alpha)^{-1} \mathbf{P}_{\pi}^* \vec{C}_{\pi} + \mathbf{H}_{\pi}\vec{C}_{\pi}+ \mathcal{O}(|1-\alpha|) \\
    & = & (1-\alpha)^{-1} \vec{J}_{\pi} + \vec{h}_{\pi}+\mathcal{O}(|1-\alpha|). \label{eq: laurent series 1}
\end{eqnarray}
The truncated Laurent Series is useful in the study of Blackwell optimality as in section 5.1.2 of \cite{BertsekasVol2} and chapter 10 of \cite{puterman2014markov}. This paper acknowledges its importance but directs interest to $\mathbf{H}(\alpha)$ which is defined as 
\begin{eqnarray}
    \mathbf{H}_{\pi}(\alpha) & = & (1-\alpha)^{-1}\mathbf{M}_{\pi}(\alpha) - (1-\alpha)^{-1}\mathbf{P}_{\pi}^* \nonumber \\
    & = & (\mathbf{I} - \alpha \mathbf{P}_{\pi})^{-1} - (1-\alpha)^{-1}\mathbf{P}_{\pi}^* \label{step: lhs}\\
    & = & \sum_{k=0}^{\infty} \alpha^k \mathbf{P}_{\pi}^k - \left( \sum_{k=0}^{\infty} \alpha^k \right)\mathbf{P}_{\pi}^* \nonumber  \\
    & = & \sum_{k=0}^{\infty} \alpha^k \left( \mathbf{P}_{\pi}^k - \mathbf{P}_{\pi}^*  \right) \nonumber \\
    & = & \mathbf{I} - \mathbf{P}_{\pi}^* + \sum_{k=1}^{\infty} \left( \mathbf{P}_{\pi} - \mathbf{P}_{\pi}^*  \right)^k \label{step: k} \\
    & = & - \mathbf{P}_{\pi}^* + \sum_{k=0}^{\infty} \left( \mathbf{P}_{\pi} - \mathbf{P}_{\pi}^*  \right)^k \nonumber  \\
    & = & \left( \mathbf{I} - \alpha \mathbf{P}_{\pi} + \alpha \mathbf{P}_{\pi}^*  \right)^{-1} + \mathbf{P}_{\pi}^* \label{step: rhs}
\end{eqnarray}
such that $\mathbf{H} = \lim_{\alpha \to 1} \mathbf{H}(\alpha)$ and step (\ref{step: k}) results from the fact that if $k \geq 1$ then  $\left( \mathbf{P} - \mathbf{P}^*  \right)^k = \mathbf{P}^k - \mathbf{P}^* $ as from (A.16) of \cite{puterman2014markov} or page 280 of \cite{BertsekasVol2}. Equating (\ref{step: lhs}) and (\ref{step: rhs}) and isolating $(\mathbf{I}-\alpha \mathbf{P})^{-1}$ a useful expression is derived
\begin{equation}
    (\mathbf{I} - \alpha \mathbf{P}_{\pi})^{-1} = \left( \mathbf{I} - \alpha \mathbf{P}_{\pi} + \alpha \mathbf{P}_{\pi}^*  \right)^{-1} + \frac{\alpha}{1-\alpha}  \mathbf{P}_{\pi}^*. \label{eq: discounted bellman inverse result}
\end{equation}
The following theorem pertains to the fundamental matrix of (\ref{eq: discounted bellman inverse result}).
\begin{theorem}
Consider the stationary distribution $\vec{\phi}_{\pi}$ such that $\vec{\phi}_{\pi} =  \mathbf{P}_{\pi}^T\vec{\phi}_{\pi}$ or $\left(\vec{\phi}_{\pi}\right)^T = \left(\vec{\phi}_{\pi}\right)^T \mathbf{P}_{\pi}$ then
\begin{equation}
    \left(\vec{\phi}_{\pi}\right)^T \left( \mathbf{I} - \alpha \mathbf{P}_{\pi} + \alpha \mathbf{P}_{\pi}^*  \right)^{-1} = \left(\vec{\phi}_{\pi}\right)^T . \label{eq: stationary discounted fundamental matrix}
\end{equation}
This result is found in equation (2.35) of \cite{cao_book} but never derived. This paper does so below as to be able to continue the discussion.
\end{theorem}
\begin{proof}
\begin{eqnarray}
    \left(\vec{\phi}_{\pi}\right)^T \left( \mathbf{I} - \alpha \mathbf{P}_{\pi} + \alpha \mathbf{P}_{\pi}^*  \right)^{-1} & = & \left(\vec{\phi}_{\pi}\right)^T \sum_{k=0}^{\infty} \alpha^k \left( \mathbf{P}_{\pi} - \mathbf{P}_{\pi}^*  \right)^k \nonumber \\
    & = & \left(\vec{\phi}_{\pi}\right)^T \left(  \mathbf{I} +  \sum_{k=1}^{\infty}  \mathbf{P}_{\pi}^k - \mathbf{P}_{\pi}^*   \right) \nonumber \\
    & = & \left(\vec{\phi}_{\pi}\right)^T + \sum_{k=1}^{\infty} \alpha^k \left( \left(\vec{\phi}_{\pi}\right)^T\mathbf{P}_{\pi} - \left(\vec{\phi}_{\pi}\right)^T\mathbf{P}_{\pi}^*  \right) \nonumber \\
    & = &  \left(\vec{\phi}_{\pi}\right)^T + \sum_{k=1}^{\infty} \alpha^k \left( \left(\vec{\phi}_{\pi}\right)^T - \left(\vec{\phi}_{\pi}\right)^T \right) \nonumber \\
    & = & \left(\vec{\phi}_{\pi}\right)^T\nonumber 
\end{eqnarray}
\end{proof}

The desired relationship between the stationary system-based performance of the discounted and average cost scenarios can now be stated and proved.

\begin{theorem}\label{theorem: main}
The stationary system-based performance of discounted-cost MDP is related to the stationary system-based performance of the average-cost MDP under the same policy through a scalar that depends only on the discount factor
\begin{equation}
    \eta_{\pi}^\alpha = \frac{\eta_\pi}{1-\alpha} \label{eq: result}
\end{equation}
given that the $\mathbf{P}_{\pi}$ is uni-chain.
\end{theorem}
\begin{proof}
\begin{align}
    \eta_{\pi}^\alpha & =  \left(\vec{\phi}_{\pi}\right)^T \vec{J}_{\pi}^\alpha \nonumber \\
    & =  \left(\vec{\phi}_{\pi}\right)^T (\mathbf{I} - \alpha \mathbf{P}_{\pi})^{-1} \vec{C}_{\pi} \nonumber \\
    & =  \left(\vec{\phi}_{\pi}\right)^T\left( \mathbf{I} - \alpha \mathbf{P}_{\pi} + \alpha \mathbf{P}_{\pi}^*  \right)^{-1}\vec{C}_{\pi} + \frac{\alpha}{1-\alpha} \left(\vec{\phi}_{\pi}\right)^T \mathbf{P}_{\pi}^* \vec{C}_{\pi}&  \mbox{[eq (\ref{eq: discounted bellman inverse result})]} \nonumber \\
    & =  \left(\vec{\phi}_{\pi}\right)^T \vec{C}_{\pi} + \frac{\alpha}{1-\alpha} \left(\vec{\phi}_{\pi}\right)^T\vec{C}_{\pi}& \mbox{[eq (\ref{eq: stationary discounted fundamental matrix})]} \label{eq: line 0}\\
    & = \frac{1}{1-\alpha} \left(\vec{\phi}_{\pi}\right)^T \vec{C}_{\pi}\\
    & =  \frac{J}{1-\alpha} & \mbox{[eq (\ref{eq: scalar J})]} \label{eq: line 1} \\
    & = \frac{\eta_{\pi}}{1-\alpha} & \mbox{[eq (\ref{eq: eta equal J})]} \label{eq: line 2}
\end{align}
\end{proof}
\begin{remark}
Steps (\ref{eq: line 1}) and (\ref{eq: line 2}) require the uni-chain assumption. Recall that for a uni-chain with transient states, $J_{\pi^*} = \vec{\phi}_{\pi^*} \vec{C}_{\pi^*}$ but $\pi^*$ is not necessarily biased-optimal whereas a uni-chain that consists only a single recurrent class is.
\end{remark}
\begin{corollary}\label{corollary: main}
Equation~(\ref{eq: line 0}) makes no uni-chain assumption and can be used when $\mathbf{P}_{\pi}$ is \emph{multi-chain}.
\end{corollary}

The following lemma will prove useful in the next section where it can be used in conjunction with Blackwell optimality as to identify when discounted state-based optimally implies discounted stationary system-based optimality. This \emph{comparison lemma} follows directly from theorem~\ref{theorem: main} and equation~(\ref{eq: line 2})

\begin{lemma}\label{lemma: comparison}
For two policies $\pi,\pi' \in \Pi$ that induce uni-chains ($\mathbf{P}_{\pi}$ and $\mathbf{P}_{\pi'}$) a state-based partial ordering (if it exists)
\begin{eqnarray}
    \vec{J}_{\pi}^\alpha \preceq \vec{J}_{\pi'}^\alpha \implies \eta_{\pi}^\alpha <(=)\,\eta_{\pi'}^\alpha
\end{eqnarray}
if and only if $\eta_{\pi} <(=)\,\eta_{\pi'}$. Hence, a policy's stationary discounted system-based rank is determined by its stationary average-cost system-based rank.
\end{lemma}


\subsection{Blackwell optimality}\label{section: blackwell optimal}

Blackwell optimality is a well-known conceptual link between average-cost and discounted-cost MDPs (see chapter 5.1.2 of \cite{BertsekasVol2}). It is motivated by the growing dominance of $(1-\alpha)\vec{J}_\pi^\alpha$ in the truncated Laurent series expansion~(\ref{eq: laurent series 1}) as $\alpha$ approaches one. The truncated Laurent
\begin{equation}
    \vec{J}_{\pi} = (1-\alpha)\vec{J}_{\pi}^\alpha - (1-\alpha)\vec{h}_{\pi} + \mathcal{O}\left(|1-\alpha|^2\right) \label{eq: laurent series 2}
\end{equation}
This suggests that average-cost and discounted-cost policies should become more alike as $\alpha$ increases. The definition of a Blackwell optimal policies follows below.

\begin{definition}[\textbf{Blackwell optimal policy} \cite{BertsekasVol2}]
For some $\bar{\alpha} \in (0,1)$ a \emph{Blackwell optimal policy} $\bar{\pi} \in \Pi$ is such that it is \emph{both} optimal for the average-cost problem and all discounted-cost problems in the interval $(\bar{\alpha},1)$.
\end{definition}
The following characteristics are discussed in chapter 5.1.2 of \cite{BertsekasVol2}:
\begin{itemize}
    \item There may exist a stationary Blackwell optimal policy for a given infinite-horizon MDP (c.f proposition 5.1.3). However, a uni-chain average-cost optimal policy need not be Blackwell optimal as noted on page 304 of \cite{BertsekasVol2}.
    \item Among the class of stationary policies $\pi \in \Pi$, $\bar{\pi}$ is optimal such that $\forall \pi \in \Pi\setminus\{\bar{\pi}\}: \, \vec{J}_{\bar{\pi}} \preceq \vec{J}_{\pi}$ where strict inequality holds at least once.
    \item Not all optimal stationary policies are Blackwell optimal. This is noticeably true for the average-cost case where $\pi^*$ is not unique as in the discounted case. Instead, it may have a set of gain-optimal policies $\Pi_2$ (as discussed on page~\pageref{text: gain optimal set}) from which one could be Blackwell optimal.
    \item Blackwell optimal policies are stationary optimal but also minimise transient costs to some degree. Intuitively this is because the $\bar{\alpha}$-discounted MDP addresses transient behaviour starting form each state $x\in \mathcal{X}$ in order to achieve state-based optimality. The lower $\bar{\alpha}$ is, the more likely transience is optimised for.
    \item Blackwell optimality exists for non-stationary policies as well $\pi \in \Pi_t$. Note that in this paper it is assumed that $\Pi\cap \Pi_t= \varnothing$ and only $\Pi$ is considered.
    \item Blackwell optimal policies are optimal over all terms of the full Laurent series expansion. In other words, it is $m$-\emph{discount optimal} if it satisfies the following
    \begin{equation}
        \forall \pi \in \Pi\setminus\{\bar{\pi}\},\,\forall k \in \mathbb{Z}\cap[-1,m) : \quad  \lim_{\alpha \to 1} (1-\alpha)^{-k} \left( \vec{J}_{\bar{\pi}}  \ominus \vec{J}_{\pi}  \right) \preceq \vec{0} \label{eq: n bias optimality}
    \end{equation}
    where $\ominus$ denotes element-wise subtraction and $\vec{0}$ is an appropriate size column vector of zeros.  Note that a $m=-1$ optimal policy corresponds to a gain-optimal policy while $m=0$ optimality refers to a bias-optimal policy. Moreover, $m=\infty$ optimality is equivalent to Blackwell optimality. Uni-chain MDPs are guaranteed to have a 0-optimal policy.
\end{itemize}

With the understanding that a discounted-cost MDP has a unique optimal policy and that a Blackwell-optimal policy optimises both the the discounted and average-cost criteria, the following conclusion holds.
\begin{proposition}\label{prop: one blackwell optimal policy}
If a Blackwell optimal policy exists for a finite state and finite action MDP then there can only be one.
\end{proposition}
\begin{proof}
Consider an ordered set $\Pi_\alpha = \{ \pi^*(\alpha): \alpha \in (0,1) \}$ where the last entry is $\pi^*(\bar{\alpha})$. Each $\pi^*(\alpha)$ is the unique optimal policy for its discounted-cost MDP. Similarly consider $\Pi_2$ for the average-cost MDP as discussed on page~\pageref{text: gain optimal set} where $|\Pi_2| \geq 1$. Then if a Blackwell optimal policy exists $\Pi_\alpha \cap \Pi_2 = \{ \pi^*(\bar{\alpha})\}$ otherwise $\Pi_\alpha \cap \Pi_2 = \varnothing$.
\end{proof}

It can be said that for a Blackwell optimal policy, state-based optimality must hold for both the average-cost and $\bar{\alpha}$-discounted MDP. By theorem~\ref{theorem: average costs implies stationary system optimality}, stationary system-based optimality is obtained for the average-cost case. System-based optimality for the discounted MDP is related to it by a constant of proportionality $(1-\alpha)^{-1}$ where $\alpha \in (\bar{\alpha},1)$ through theorem~\ref{theorem: main}. This supports the following results.

By theorem~\ref{theorem: main} and lemma~\ref{lemma: comparison} the desired result for stationary discounted system-based optimality is achieved.

\begin{corollary}\label{corollary: blackwell system-based optimality}
If Blackwell optimal policy $\bar{\pi} \in \Pi_2$ exists and it induces a uni-chain $\mathbf{P}_{\bar{\pi}}$ the following holds $\forall \alpha \in (\bar{\alpha},1)$
\begin{eqnarray}
    \forall \pi \in \Pi\setminus\Pi_2: \quad  \eta_{\bar{\pi}} < \eta_{\pi} & \implies & \forall \pi \in \Pi\setminus\Pi_2: \quad \frac{\eta_{\bar{\pi}}}{1-\alpha} < \frac{\eta_{\pi}}{1-\alpha}\\
    & \implies & \forall \pi \in \Pi\setminus\Pi_2: \quad  \eta_{\bar{\pi}}^{\alpha} < \eta_{\pi}^{\alpha}
\end{eqnarray}
where strict inequality follows from the fact that $\Pi_2$ is the set of gain-optimal policies.
\end{corollary}
In investigating $\Pi_2$, proposition~\ref{prop: one blackwell optimal policy} plays an additional role
\begin{corollary}\label{corollary:non-unique}
If Blackwell optimal policy $\bar{\pi} \in \Pi_2$ exists where $|\Pi_2|>1$ such that it induces a uni-chain $\mathbf{P}_{\bar{\pi}}$ then the following holds $\forall \alpha \in (\bar{\alpha},1)$
\begin{eqnarray}
    \forall \pi \in \Pi_2: \quad  \eta_{\bar{\pi}} & = & \eta_{\pi}\\
    \forall \pi \in \Pi_2: \quad  \eta_{\bar{\pi}}^\alpha & = & \eta_{\pi}^\alpha
\end{eqnarray}
with the condition that
\begin{equation}
    \forall \pi \in \Pi_2: \quad \vec{J}_{\bar{\pi}}^\alpha \preceq \vec{J}_{\pi}^\alpha
\end{equation}
where strict inequality holds at least once. This partial order follows from proposition~\ref{prop: one blackwell optimal policy}.
\end{corollary}
State-based optimality can now be related to system-based optimality for the stationary discounted case.
\begin{corollary}\label{corollary: MAIN}
If and only if an optimal discounted-cost MDP $\pi^* \in \Pi_{\alpha}\cap \Pi_2$ is \emph{both} Blackwell optimal and induces a uni-chain $\mathbf{P}_{\pi^*}$ does the system possess the following:
\begin{eqnarray}
    \forall \pi \in \Pi\setminus\{\pi^*\}:\quad \vec{J}_{\pi^*}^\alpha \preceq \vec{J}_{\pi}^\alpha & \implies & 
    \begin{cases}
    \forall \pi \in \Pi\setminus\{\pi^*\}:\quad \eta_{\pi^*}^\alpha < \eta_{\pi}^\alpha & \pi \not \in \Pi_2\\
    \forall \pi \in \Pi\setminus\{\pi^*\}:\quad \eta_{\pi^*}^\alpha = \eta_{\pi}^\alpha & \pi \in \Pi_2.
    \end{cases}
    \label{eq: discounted state based optimality system based optimality}
\end{eqnarray}
If the uni-chain consist of a single recurrent class then $|\Pi_2|=1$ and only the first case of strict inequality holds.
\end{corollary}

Hence, a sufficient condition has been identified under which discounted state-based optimality ensures discounted stationary system-based optimality. All that remains is to determine whether a policy is Blackwell optimal. For finite state and finite action MDPs, section 2.5.2 of \cite{MDP_handbook} discusses and presents a Policy Iteration algorithm for finding $n$-discount optimal policies. For sufficiently large $n$, one might assume the policy to be Blackwell optimal. However, section 2.5.3 of \cite{MDP_handbook} discusses a Linear Programming approach that finds a Blackwell optimal policy through tuning the discount factor. 

This paper proposes a practical approach. By the nature of the problem at hand, the discounted MDP already has its optimal policy solved for such that $\forall \pi \in \Pi_\alpha\setminus\{\pi^*\}:\, \vec{J}_{\pi^*}^\alpha \preceq \vec{J}_{\pi}^\alpha $. Hence, it only needs to be verified that $\pi^* \in \Pi_2$. This can be done by running average-cost Policy Iteration starting from $\pi^*$. In the first iteration, \emph{policy evaluation} will compute its bias $J_{\pi^*}$. Then \emph{policy improvement} returns $\pi^{**}$. If $\pi^* = \pi^{**}$, one can confirm $\pi^* \in \Pi_2$ by Policy Improvement theorem \cite{suttonRLbook}. Otherwise, $\pi^* \not\in \Pi_2$. This follows because $\pi^**$ cannot have a better bias than $\pi^*$ if the latter were Blackwell optimal. This can be seen from $n$-bias optimality relation (\ref{eq: n bias optimality}) where $m=\infty$ is Blackwell optimal. This can be confirmed by running policy evaluation for $\pi^**$ which should return $J_{\pi^**} < J_{\pi^*}$ thus reaffirming the claim $\pi^* \not\in \Pi_2$.

\subsection{Example: random MDPs}

A simple means of assessing equation~(\ref{eq: result}) is to generate random MDPs with optimal policies that are guaranteed to induce a uni-chain $\mathbf{P}_{\pi^*}$. For an MDP with $|\mathscr{A}|$ actions and $|\mathcal{X}|$, each $\mathbf{P}_a \in \mathcal{P}$ can be constructed row-wise by sampling its $i^{th}$ row from the surface of a $|\mathcal{X}|$-dimensional simplex $(\mathbf{P}_a)_i = p_i \sim \mathbf{\Delta}^{|\mathcal{X}|}$ where $p_i = [p_{i,1},\cdots,p_{i,|\mathcal{X}|} ]$. Sampling from such a simplex is equivalent to sampling from a \emph{Dirichlet distribution} $\mathcal{D}\left(\vec{\theta}\,\right)$ where $\vec{\theta} = [\theta_1,\cdots,\theta_{|\mathcal{X}|}]$ and $\forall \theta_i \in \vec{\theta}: \, \theta_i > 0$. Its probability density function is given as
\begin{equation}
    f_\theta(p_i) = \frac{1}{\beta\left(\vec{\theta}\,\right)} \prod_{j=1}^{|\mathcal{X}|} \left(p_{i,j}\right)^{\theta_j-1}
\end{equation}
with a normalising constant
\begin{equation}
    \beta\left(\vec{\theta}\,\right) = \frac{\prod_{j=1}^{|\mathcal{X}|} \Gamma(\theta_j)}{\Gamma\left( \sum_{j=1}^{|\mathcal{X}|} \theta_j  \right)}
\end{equation}
where $\Gamma$ is the gamma function
\begin{equation}
    \Gamma(z) = \int_0^\infty  x^{z-1} e^{-x} \, dx .
\end{equation}
Sampling from a Dirichlet distribution can be performed by sampling from Gamma distributions $F_{\Gamma}$ which in turn can be sampled using the inverse-transform method \cite{devroye_random_numbers}. A vector $\vec{y} = [y_1,\cdots,y_{|\mathcal{X}|}]$ is sampled via $y_j \sim F_{\Gamma}(\theta_j)$. The desired sample $p_i$ results from normalising $\vec{y}$ by its sum. It is rare that a zero entry should be found. If it is desired to safeguard against such a scenario then a small positive constant can be added to all entries in a row. This row can be re-normalised by dividing all entries by its sum. If $\forall \theta_j \in \vec{\theta}: \theta_j = \theta$ where $\theta > 0$ then $p_i$ is \emph{uniformly} sampled from the surface of $\mathbf{\Delta}^{|\mathcal{X}|}$.

If all $\mathbf{P}_a \in \mathcal{P}$ have only non-zero entries then any policy $\pi \in \Pi$ induces a uni-chain $\mathbf{P}_{\pi}$ with a single recurrent class. Hence, all results that pertain to this type of chain hold such as a unique gain-optimal policy. To ensure that a uni-chain with a single recurrent class and some transient states is always induced by any policy the following modifications can be made. If $N_t$ transient states is desired then there should be $N_r = |\mathcal{X}|-N_t$ states in the closed recurrent class. The for all $\mathbf{P}_a \in \mathcal{P}$ a block of zeros is introduced as to isolate the recurrent class by setting 
$$
\left(\mathbf{P}_a\right)_{[1:N_r],[N_r+1:|\mathcal{X}|]} = 0 
$$
and re-normalising the first $N_r$ rows. Hence the first $N_r$ states form the recurrent class. The costs vector $\vec{C}_a \in \mathcal{C}$ can be sampled entry-wise from some arbitrary distribution such as a bounded uniform distribution $\left( \vec{C}_a \right) \sim Uni(a,b)$.

The sampled MDP can then have its average and discounted cost policies solved for using Policy iteration as described in section~\ref{section: average cost solve} and section~\ref{section: discounted cost solve}, respectively. It should be noted than increasing $N_t$ results in a larger $\Pi_2$ if exploring gain-optimal policies is of interest. 

The following five state uni-chain MDP with three transient states was sampled:
\begin{equation}
    \mathcal{P} = \left\{
    \left[\begin{matrix}0.0759 & 0.9241 & 0.0 & 0.0 & 0.0\\0.089 & 0.911 & 0.0 & 0.0 & 0.0\\0.0274 & 0.5755 & 0.2827 & 0.0127 & 0.1018\\0.0262 & 0.6042 & 0.2709 & 0.0256 & 0.0731\\0.0449 & 0.5322 & 0.2967 & 0.0376 & 0.0885\end{matrix}\right] \, , \, \left[\begin{matrix}0.4243 & 0.5757 & 0.0 & 0.0 & 0.0\\0.4474 & 0.5526 & 0.0 & 0.0 & 0.0\\0.1825 & 0.2535 & 0.1553 & 0.2027 & 0.2062\\0.1656 & 0.2366 & 0.1779 & 0.2106 & 0.2093\\0.1489 & 0.2904 & 0.1818 & 0.1544 & 0.2245\end{matrix}\right]
    \right\} \nonumber
\end{equation}
\begin{equation}
    \mathcal{C} = \left\{    
    \left[\begin{matrix}5.5386\\1.4692\\7.6187\\7.9702\\5.2197\end{matrix}\right] \, , \, \left[\begin{matrix}9.7115\\1.6244\\7.7807\\4.4739\\9.9177\end{matrix}\right]
    \right\} \nonumber
\end{equation}
such that $|\Pi_2| = 2$. More specifically, $\Pi_2 = \{ \vec{0},\vec{e}_4 \}$ where $\vec{0}$ is a vector of zeros and $\vec{e}_4$ is a unit vector with a one as its fourth entry. The latter is Blackwell optimal. Hence, corollary~\ref{corollary: MAIN} applies to any discounted policy that matches it. The following table illustrates the comparison lemma~\ref{lemma: comparison} and use of (\ref{eq: result}). Note that $\eta_\pi^\alpha$ has been obtained through (\ref{eq: result}) and can be verified to equal $(\vec{\phi}_\pi)^T \vec{J_\pi^\alpha}$ as $\vec{\phi}_\pi$ has been included for each policy. The non-uniqueness result (c.f. corollary~\ref{corollary:non-unique}) of $\eta_{\pi}^\alpha$ for a uni-chain with $|\Pi|_2 > 1$ despite $\vec{J}_{\bar{\pi}}^\alpha \preceq \vec{J}_{\pi^*}^\alpha$ is evident in comparing table~\ref{tab:random mdp blackwell optimal} with table~\ref{tab:random mdp gain optimal}. A sub-optimal policy $\pi = \mathbf{1}$ has been included as to illustrate both cases found in corollary~\ref{corollary: MAIN}.

\begin{table}[!htbp]
    \centering
    \begin{tabular}{|c|c|c|c|c|c|c|c|}
    \hline
    $\alpha$ & $J_\pi^\alpha(1)$ & $J_\pi^\alpha(2)$ & $J_\pi^\alpha(3)$ & $J_\pi^\alpha(4)$ & $J_\pi^\alpha(5)$& $\eta_\pi$ & $\eta_\pi^\alpha$\\
    \hline
    0.20 &5.9856&  1.9268&  8.4924 &  5.5562&  6.1328 & 1.8267 & 2.2834\\

    0.50 & 7.3411 &   3.2982 & 10.6515 &  8.1534&   8.3658 & 1.8267  &  3.6534\\
    
    0.75 & 10.9826&  6.9528& 15.1808& 13.2228& 12.9771 & 1.8267 & 7.3068\\
    
    0.99 & 186.3337& 182.3164& 191.6983& 190.4837& 189.5687 & 1.8267 & 182.6700\\
    \hline
    \end{tabular}
    \caption{Blackwell optimal $\bar{\pi} = \vec{e}_4$ and $\vec{\phi} = [0.0878, 0.9122, 0    , 0    , 0    ]$}
    \label{tab:random mdp blackwell optimal}
\end{table}

\begin{table}[!htbp]
    \centering
    \begin{tabular}{|c|c|c|c|c|c|c|c|}
    \hline
    $\alpha$ & $J_\pi^\alpha(1)$ & $J_\pi^\alpha(2)$ & $J_\pi^\alpha(3)$ & $J_\pi^\alpha(4)$ & $J_\pi^\alpha(5)$& $\eta_\pi$ & $\eta_\pi^\alpha$\\
    \hline
    0.20 &    5.9856 &    1.9268 &    8.5018 &    8.8303 &    6.1584 &  1.8267 &    2.2834 \\
  0.50 &    7.3411 &    3.2982 &   10.6757 &   10.9570 &    8.4247 &  1.8267 &    3.6534 \\
  0.75 &   10.9826 &    6.9528 &   15.2149 &   15.4400 &   13.0522 &  1.8267 &    7.3068 \\
  0.99 &  186.3337 &  182.3164 &  191.7318 &  191.8645 &  189.6358 &  1.8267 &  182.6700 \\
    \hline
    \end{tabular}
    \caption{Gain optimal $\pi^* = \vec{0}$ and $\vec{\phi} = [0.0878, 0.9122, 0    , 0    , 0    ]$}
    \label{tab:random mdp gain optimal}
\end{table}

\begin{table}[!htbp]
    \centering
    \begin{tabular}{|c|c|c|c|c|c|c|c|}
    \hline
    $\alpha$ & $J_\pi^\alpha(1)$ & $J_\pi^\alpha(2)$ & $J_\pi^\alpha(3)$ & $J_\pi^\alpha(4)$ & $J_\pi^\alpha(5)$& $\eta_\pi$ & $\eta_\pi^\alpha$\\
    \hline
    0.20 &   10.9808 &    2.9309 &    9.3375 &    6.0425 &   11.4554 &  5.1609 &    6.4511 \\
0.50 &   14.8204 &    6.8257 &   13.7655 &   10.4974 &   15.8534 &  5.1609 &   10.3218 \\
0.75 &   25.1166 &   17.1673 &   24.8065 &   21.5745 &   26.8658 &  5.1609 &   20.6435 \\
 0.99 &  520.5365 &  512.6302 &  521.5189 &  518.2543 &  523.4517 &  5.1609 &  516.0876 \\
    \hline
    \end{tabular}
    \caption{Sub-optimal $\pi = \mathbf{1}$ and $\vec{\phi} = [0.4373, 0.5627, 0    , 0    , 0    ]$}
    \label{tab:random mdp not optimal}
\end{table}

\newpage

\section{Application: queue admission control}

A queue might be subject to high traffic intensity $\rho = \mu/\lambda \approx 1$ where $\mu \in \mathbb{R}_{>0}$ is the service rate and $\lambda \in \mathbb{R}_{>0}$ is the arrival rate. It is known that $\rho<1$ is a sufficient condition for queue stability \cite{stewart2009probability} such that the queue will not grow without bound. With high traffic intensity, queue lengths can still fluctuate tremendously. This is an issue if the buffers are of finite size . Additionally, such variable queue lengths induce a large variance in the holding costs which can be interpreted as risk. These two issues can be mitigated by a policy that rejects customer arrivals based on the current length of the queue. The rejected customer results in a penalty due to lost business. Hence such a policy must optimise the system through balancing long-term holding and rejection costs.

This section studies such a controlled $M/M/1$ queue where an existing policy is in operation $\pi_0 \in \Pi$. An investigation has been ordered to assess whether the policy should be updated. More specifically, management is concerned with whether a new policy would provide a \emph{statistically significant} performance improvement as opposed to some theoretical one due to the cost of updating the policy. As such, an average-reward MDP policy $\bar{\pi} \in \Pi$, a discounted MDP policy $\pi_{\alpha} \in \Pi$ and Blackwell-optimal MDP policy $\pi_{\bar{\alpha}} \in \Pi$ are solved for. Empirical cost-distributions for their stationary and uniform system-based performance are obtained through simulation. These are subsequently used to determine whether the null hypothesis of \emph{no performance gain over $\pi_0$} can be rejected.

\subsection{Model}

The state-space is given by the queue length $x \in \mathcal{X} = \mathbb{Z} \cap [0,N]$ where $N$ is some finite positive integer used to denote the maximum queue length. Note that $x$ contains both the customers in the buffer as well as the customer receiving service. Each customer incurs a holding cost at rate $c \in \mathbb{R}_{0}$. The inter-arrival duration of customers follows an exponential distribution $t_{\lambda} \sim Exp(\lambda)$ as well as the service durations $t_{\mu} \sim Exp(\mu)$. These two event processes are independent. A policy $\pi: \mathcal{X} \to \{0,1\}$ prescribes an action $a=\pi(x)$ on whether to accept an arrival $a=1$ or to reject it $a=0$ were such an event to occur before a service completion. Only in the largest state is rejection the only permissible action $\mathscr{A}(N) = \{0\}$ such that $\forall x \in \mathcal{X}\setminus\{N\}: \mathscr{A}(x) = \{0,1\}$ where $\mathscr{A}(x)$ is the set of feasible actions from which a policy can be constructed. If a customer is rejected then a lumps-um penalty $R \in \mathbb{R}_{>0}$ is instantaneously incurred. The complete MDP models for the average and discounted costs follow through specifying the cost vectors, transition models and Bellman equations obtained through uniformisation of the underlying \emph{Continuous Time Markov Chain} (CTMC) \cite{BertsekasVol2,cassandras_book}.

\subsubsection{Average-cost model}\label{section: average cost solve}

Uniformisation is performed by sampling the system at some global rate $\gamma \geq  \min\{\lambda,\mu\}  = \lambda + \mu$ such that an event is never missed where the summation follows as a property of the exponential distribution \cite{harchol2013performance}. The duration of an event interval follows as $t_{\gamma} \sim Exp(\gamma)$ such that $\mathbbm{E}[t_{\gamma}] = 1/\gamma$. Hence, a holding cost of $(cx)/\gamma$ is always incurred upon entering state $x$. An additional rejection penalty can be added such that $(cx)/\gamma + \mathscr{P}(\lambda)R$ follows for $a(x)=0$ where $\mathscr{P}(\lambda)$ is the probability of an arrival and is defined in the next paragraph. A set of cost vectors is constructed for each decision $\mathcal{C} = \{ \vec{C}_0,\vec{C}_1 \}$ where $|\vec{C}_i| = N+1$ as to include $x=0$. Furthermore, $\vec{C}_{0} = [(cx)/\gamma+ \mathscr{P}(\lambda)R: x \in \mathcal{X}]$ and $\vec{C}_{1} = [(cx)/\gamma : x \in \mathcal{X}]$.

The transition model can be constructed from the fact that the probability of an arrival $\mathscr{P}(\lambda) =  \lambda/\gamma$ and of a service completion $\mathscr{P}(\mu) = \mu/\gamma$ are known. To address the case of no possible service event when $x=0$, a fictitious self-transition event $\Theta$ is introduced such that $\mathscr{P}(\Theta) = \mu/\gamma$. Hence a set of $(N+1)\times (N+1)$ transition matrices can be constructed $\mathcal{P} = \{ \mathbf{P}_0,\mathbf{P}_1\}$.

Using integer indexing that starts counting from zero, $\mathbf{P}_0$ is constructed row-wise 
\begin{equation}
    (\mathbf{P}_0)_{i,j} = 
    \begin{cases}
    \frac{\lambda}{\gamma}, & j=i\\
    \frac{\mu}{\gamma},& j = \max\{i-1,0\}
    \end{cases}
\end{equation} 
and $\mathbf{P}_1$ follows 
\begin{equation}
    (\mathbf{P}_1)_{i,j} = 
    \begin{cases}
    \frac{\lambda}{\gamma}, & j = \min\{ N,i+1 \}  \\
    \frac{\mu}{\gamma},& j = \max\{i-1,0\}.
    \end{cases}
\end{equation}
Note that $j = \min\{ N,i+1 \}$ was included to maintain $\mathbf{P}_1$ as a stochastic matrix. This is only for aesthetic purposes because $\mathcal{A}(N) = \{0\}$ such that $\mathbf{P}_1$ will never have its last row consulted when $x=N$.

In the bellman equations, an average-cost or gain will be sustained over the waiting time of each state $\bar{t}_\gamma = \mathbb{E}[t_\gamma] = 1/\gamma$. The waiting time is the same in all states. Furthermore, the MDP is uni-chain such that the gain is the same for all states. This gain is denoted as $J_{\pi}^\gamma = J_\pi\times \bar{t}_\gamma$ where $J_\pi$ can is interpreted as the average-cost per unit time and is naturally a rate. The uniformised average-cost Bellman equations are presented in matrix notation below
\begin{equation}
    \vec{h}_\pi + J_{\pi}^\gamma\mathbf{1} = \vec{C}_\pi + \mathbf{P}_\pi\label{eq: uniformised average-cost Bellman equations}
\end{equation}
as a system of $N+1$ equations with $N+2$ unknowns. To mitigate this issue a distinguished state/index is randomly selected from a uniform distribution\footnote{It need not be randomly selected and could be chosen as $i^\# = 0$. It was uniformly selected to highlight that any state would suffice with non being preferential.} $i^\# \sim Uni(\mathbb{Z}\cap[0,N+1])$. A $N+1$ length unit vector $\vec{e}_{i^\#}$ is to be constructed with a one at the distinguished index. A augmented system of $N+2$ equations can be solved for the purposes of policy evaluation by including the restriction that $h(i^\#)=0$

\begin{equation}
    \left[\begin{array}{c;{2pt/2pt}c}
        \mathbf{I} - \mathbf{P}_\pi  & \mathbf{1} \\
        \hdashline[2pt/2pt]
        \vec{e}_{i^\#} & 0
    \end{array}  \right]
    \left[\begin{array}{c}
         \vec{h}_\pi \\ 
         \hdashline[2pt/2pt]
         J_\pi^\gamma
    \end{array}\right] = 
    \left[\begin{array}{c}
         \vec{C}_\pi \\ \hdashline[2pt/2pt]
         0
    \end{array}\right]. 
\end{equation}
Such policy evaluation when paired with standard policy improvement as below produces a valid Policy Iteration algorithm that converges at iteration $k+1$ once $\forall x \in \mathcal{X}: \, \vec{\pi}_k(x) = \vec{\pi}_{k+1}(x) $
\begin{equation}
    \vec{\pi}_{k+1}(x) = \operatorname{argmin}_{a \in \mathscr{A}(x)} \left\{  \vec{C}_a(x) + \sum_{x' \in \mathcal{X}} \mathbf{P}_a(x'\mid x) \vec{h}_{k}(x')  \right\}
\end{equation}
where $\vec{h}_{k}$ are the evaluated relative biases of the $k^{th}$ policy.

\subsubsection{Discounted-cost model}\label{section: discounted cost solve}
The discounted cost model retains the same set of transition models $\mathcal{P}$ but requires a different $\mathcal{C}$ and Bellman equations as well as an additional discount factor $\alpha \in (0,1)$. The discount factor results from uniformisation and the some interest rate $\beta > 0$ and is computed as $\alpha = \gamma/(\gamma + \beta)$ \cite{cassandras_book,BertsekasVol2}. The costs per state now consists of a discounted holding cost $c_\alpha(x) = cx/(\beta + \gamma)$ and discounted rejection penalty $\alpha R$. The updated $N+1$ cost vectors are $\vec{C}_0 = [c_\alpha(x) + \alpha R: x \in \mathcal{X}] $ and $\vec{C}_1 = [c_\alpha(x): x \in \mathcal{X}] $. Policy evaluation is performed by solving the uniformised discounted-cost Bellman equation
\begin{equation}
    \vec{J}_{\pi}^\alpha = \vec{C}_{\pi} + \alpha \mathbf{P}_{\pi} \vec{J}_\pi^\alpha \label{eq: uniformised discounted-cost Bellman equation}
\end{equation}
which is a system of $N+1$ equations with $N+1$ unknowns and an invertible matrix $\mathbf{I}-\alpha\mathbf{P}_{\pi}$ such that a unique solution is guaranteed \cite{BertsekasVol2}. Policy improvement follows as 
\begin{equation}
    \vec{\pi}_{k+1}(x) = \operatorname{argmin}_{a \in \mathscr{A}(x)} \left\{  \vec{C}_a(x) + \alpha \sum_{x' \in \mathcal{X}} \mathbf{P}_a(x'\mid x) \vec{J}_{k}^\alpha(x')  \right\}
\end{equation}
which completes the Policy Iteration algorithm. Note that $\vec{J}_k^\alpha$ are the discounted state-values of the $k^{th}$ policy.

\subsection{Policies and system parameters}

It has been shown that the admission problem has a \emph{threshold/index policy} \cite{cassandras_book} where there exists $x^* \in \mathcal{X}\setminus\{ N\}$ such that
\begin{eqnarray}
    \pi^*(x) = \begin{cases}
    0, & x \geq x^* \\
    1,  &  x < x^*
    \end{cases}.
\end{eqnarray}
Hence a policy is completely characterised by this single threshold. Furthermore, it can be shown that $\Delta J(x) = J_1(x)-J_0(x)$ is a monotonically decreasing function \cite{cassandras_book} such that a numerically obtained MDP policy will discover the index structure $\Delta J(x^*) < 0$ and $\Delta J(x^*-1) \geq 0$. This will be the best index policy among all index policies according to its objective function.

For a controlled $M/M/1$ queue with a size of $N=30$, arrival rate $\lambda = 1$, service rate $\mu=0.95$, cost rate $c=1$ and rejection penalty $R=200$ an existing policy $\pi_0$ is in place with a threshold $x_0^*=17$. Three alternative policies, obtained through solving a MDP, are proposed and are tabulated below along with the existing policy.

\begin{table}[!htbp]
    \centering
    \begin{tabular}{|c|c|c|c|c|}
    \hline
        \textbf{Policy} & \textbf{Objective function} & \textbf{Threshold} & \textbf{Interest rate} $\beta$ & \textbf{Discount factor} $\alpha$ \\
        \hline
        Existing $\pi_0$ & None & 17 & None & None\\
        
        Average-cost $\bar{\pi}$ & Bias $J_{\bar{\pi}}$ & 16 & None & None\\
        
        Discounted-cost $\pi_\alpha$ & State-value $J_{\pi}^\alpha$ & 19 & $2\times 10^{-3}$ & $0.998975$ \\
        
        Blackwell-optimal $\pi_{\bar{\alpha}}$ & State-value $J_{\pi}^{\bar{\alpha}}$ & 16 & $4\times 10^{-4}$ &  $0.999795$\\
        \hline
    \end{tabular}
    \caption{Threshold/index policies}
    \label{tab: policies}
\end{table}

\subsection{Theoretical results}

In this section, the stationary and uniform system based performances are computed for each policy under each objective function. 

\subsubsection{Uniform system-based performance}
To evaluate $\nu_\pi$, $\nu_\pi^\alpha$ and $\nu_\pi^{\bar{\alpha}}$ for each policy will require obtaining $J_\pi$, $\vec{J}_{\pi}^\alpha$ and $\vec{J}_{\pi}^{\bar{\alpha}}$ through average-cost policy evaluation~(\ref{eq: uniformised average-cost Bellman equations}) and discounted-cost policy evaluation~(\ref{eq: uniformised discounted-cost Bellman equation}). Recall that (\ref{eq: uniformised average-cost Bellman equations}) returns $J_{\pi}^\gamma$ as the gain such that $J_\pi = J_{\pi}^\gamma/\bar{t}_{\gamma}$. From (\ref{eq: nu equal J}) the desired result follows $\nu_\pi = J_\pi$.

\begin{table}[!htbp]
    \centering
    \begin{tabular}{|c|c|c|c|}
    \hline
    \textbf{Policy} & $\nu_\pi$ & $\nu_\pi^\alpha$ & $\nu_\pi^{\bar{\alpha}}$ \\
    \hline
    $\pi_0$     &  26.451004 &     14111.497973 &      67035.19422 \\
    
    $\bar{\pi}$ &  \textbf{26.401347} &     14198.259172 &     \textbf{67024.011784} \\
    
    $\pi_\alpha$   &  26.764367 &      \textbf{14038.44897} &     67584.439005 \\
    
    $\pi_{\bar{\alpha}}$    &  \textbf{26.401347} &     14198.259172 &     \textbf{67024.011784} \\
    \hline
\end{tabular}
    \caption{Uniform system-based performances}
    \label{tab:Uniform system-based performances}
\end{table}
From the fact that state-based optimality guarantees uniform system-based optimality, it comes as no surprise that each MDP policy performed best under the objective function it was solved for. The existing policy was subsequently outperformed by at least one policy under each criteria. Another obvious result comes from the average-cost and Blackwell optimal policies performing identically.

\subsubsection{Stationary system-based performance}

The same approach can be taken as in the previous section where each $\eta_\pi$, $\eta_\pi^\alpha$ and $\eta_\pi^{\bar{\alpha}}$ will require  $J_\pi$, $\vec{J}_{\pi}^\alpha$ and $\vec{J}_{\pi}^{\bar{\alpha}}$. If these have already been computed then it makes sense to compute $\eta_\pi$ from it by additionally obtaining $\vec{\phi}_\pi$. However, it may be the case that the computations of the previous section had not been performed yet such that each policy only has $J$, $\vec{J}_{\pi}^\alpha$ or $\vec{J}_{\pi}^{\bar{\alpha}}$ depending under which criteria it was solved for. This is sufficient enough to compute all three types of $\eta_\pi$ given $\vec{\phi}_\pi$ and equation~(\ref{eq: result}). Furthermore, policy $\bar{\pi}$ does not require $\vec{\phi}_{\bar{\pi}}$ because $\eta_{\bar{\pi}} = J_{\bar{\pi}}$ from (\ref{eq: eta equal J}). As mentioned in the previous section, average-cost policy evaluation returns $J_{\pi}^\gamma$ such that $J_\pi = J_{\pi}^\gamma/\bar{t}_{\gamma}$. Such an approach is much faster and yields the same results.

\begin{table}[!htbp]
    \centering
    \begin{tabular}{|c|c|c|c|}
    \hline
    \textbf{Policy} & $\eta_\pi$ & $\eta_\pi^\alpha$ & $\eta_\pi^{\bar{\alpha}}$ \\
    \hline
    $\pi_0$     &  26.451004 &      13239.066531 &      66141.074184 \\
    
    $\bar{\pi}$&  \textbf{26.401347} &      \textbf{13214.212448} &      \textbf{66016.905631} \\
    
    $\pi_\alpha$   &  26.764367 &      13395.909004 &      66924.643751 \\
    
    $\pi_{\bar{\alpha}}$   &  \textbf{26.401347} &      \textbf{13214.212448} &      \textbf{66016.905631} \\
    \hline
\end{tabular}
    \caption{Stationary system-based performances}
    \label{tab:Stationary system-based performances}
\end{table}

The results from table~\ref{tab:Stationary system-based performances} are noticeably different than those of table~\ref{tab:Uniform system-based performances}. The top performer for each column was not the MDP solved for under the corresponding objective function such that state-based optimality to not always guarantee system-based optimality. However, the results of theorem~\ref{theorem: average costs implies stationary system optimality} hold in the first column pertaining to $\eta_\pi$ while corollary~\ref{corollary: blackwell system-based optimality} is confirmed in the third column of $\eta_{\pi}^{\bar{\alpha}}$.

\subsection{Simulation results}

The results of the previous section are useful in identifying a theoretically optimal system. However, as seen in tables~\ref{tab:Uniform system-based performances} and \ref{tab:Stationary system-based performances}, the performance metrics do not always differ by a substantial margin. The definition of what a substantial margin might be is subjective. This paper chooses to define a substantial margin as one that is statistically significant such that the following null hypothesis can be rejected
\begin{equation}
    H_0: \mathbbm{E}\left[f_{\pi}\right] = \mathbbm{E}\left[f_{\pi_0}\right] 
\end{equation}
in favour of the alternative
\begin{equation}
    H_A: \mathbbm{E}\left[f_{\pi}\right] < \mathbbm{E}\left[f_{\pi_0}\right] 
\end{equation}
where $f_\pi: \mathbb{R}_{\geq 0} \to \mathbb{R}_{\geq 0}$ is some probability density function (i.e. $\int_{0}^\infty f_\pi (z) \, dz=1$) of a system-based performance metric and $\pi$ is a MDP policy suggested to replace $\pi_0$. Due to the fact that $f_\pi$ has no known closed-form expression, simulation has been used to derive empirical distributions $\hat{f}_{\pi}: \mathbb{R}_{\geq 0} \to \mathbb{R}_{\geq 0}$ as to approximate it.

\subsubsection{Simulation}
The $M/M/1$ queue is a \emph{Continuous-time Markov Chain} (CTMC) for which various efficient simulation routines exist \cite{stewart2009probability}. This paper will use such a CTMC as the simulator for producing trajectories 
\begin{equation}
    \mathcal{T}_\theta(\omega,\pi,x_0) = \left\{ (x_0,\Delta t_0,e_0,\mathbbm{1}{\{e_0=, a_0\}}),\cdots,(x_\tau,\Delta t_\tau,e_\tau,\mathbbm{1}{\{e_\tau, a_\tau\}})  \right\}
\end{equation}
where $T$ is the minimum duration of the simulation such that $\sum_{i=0}^{\tau} \Delta t_{i} \geq T$, $\Delta t_\tau$ is the waiting time, $e_\tau$ is the event-type that occurred, $x_\tau$ is the queue-length and $\mathbbm{1}{\{e_\tau, a_\tau\}} = \mathbbm{1}{\{e_\tau=\lambda, a_\tau=0\}}$ is the occurrence of a rejection. The simulator is characterised by its system parameters $\theta = \{ \mu,\lambda \}$. Lastly, $\omega \in \Omega$ is the underlying random numbers used to generate the trajectory. In computer simulations, these are known if a pseudo-random number generator is used. Running different policies on the same random numbers allows for fair comparisons. Intuitively, if an extreme event was generated by $\omega_i$ then both policies would likely have their trajectories subjected to it. However, the use of common random numbers is not without problems. In comparing empirical distributions constructed from sampled trajectories, common random numbers will introduce paired-sample correlation. This will be further discussed, however, a simple solution to this lies in randomly shuffling the samples of each empirical distribution before comparing them. \\

The pseudo-code for producing a single CTMC trajectory under a given policy is presented in algorithm~\ref{algorithm: ctmc simulation}. Note that an initial state $x_0$ is fed in as an argument. Depending on whether $x_0 \sim \phi_\pi$ or $x_0 \sim U$, a trajectory is sampled that pertains to $\eta$ or $\nu$, respectively.\\

The most common methods for sampling from the exponential and categorical distributions in this algorithm is through the \emph{inverse-transform} \cite{devroye_random_numbers,stewart2009probability}. Empirical distributions are constructed from $M$ length arrays $\mathbf{C}_f$ where each entry $\mathbf{C}_f[i]$ contains the system-based performance for a sampled trajectory. Hence, $\mathbf{C}_f$ is an array of independent and identically distributed random samples. The samples in these arrays can be binned, normalised and presented as histograms. Due to the large number of distributions, this paper forgoes this and presents table~\ref{tab:nu_distributions} and \ref{tab:eta_distributions} as to describe the distributions through statistics. Using the \emph{D' Agostino's $k^2$ test} for normality (see section~\ref{section:normality test}), all distributions were rejected as Gaussian. This should be expected from the positive skewness in all distributions. Such skewness can be explained by the fact that the domain of system-based performance is constrained to $\mathbb{R}_{\geq 0}$. Extreme trajectories at the low end can only pull the distribution closer to zero whereas extreme trajectories with high cost performance can pull towards an unbounded value.

\begin{algorithm}[!htbp]
    \caption{CTMC Simulation}
    \label{algorithm: ctmc simulation}
    \begin{algorithmic}[1]
    \Procedure{SimulateTrajectory}{$x_0$,$\pi$,$T$,$\omega$,$\lambda$,$\mu$}
    \State $\tau \leftarrow 0$ \Comment{Global clock.}
    \State $ x \leftarrow x_0$ \Comment{Current state}
    \State $\gamma \leftarrow \lambda + \mu$ \Comment{Sampling rate.}
    \State $p\leftarrow [\mu/\gamma,\lambda/\gamma]$ \Comment{Event probabilities.}
    \State $ \mathcal{T} \leftarrow \{\} $ \Comment{Empty order-preserving set.}
    \While{$\tau < T$}
        \State $a= \pi(x)$ \Comment{Consult the policy.}
        \State $\Delta t \sim Exp(\gamma)$ \Comment{Waiting time.}
        \State $e \sim Cat(p)$  \Comment{Categorical distribution returns 0 or 1.}
        \If{$e=0$}
            \State $x' \leftarrow \max\{0,x-1\}$
        \Else
            \State $x' \leftarrow x + \mathbbm{1}\{a=1\}$
        \EndIf
        \State $\mathcal{R} \leftarrow \mathbbm{1}\{a=0\}$ \Comment{Check for arrival rejections.}
        \State $\mathcal{T} \leftarrow \mathcal{T} \cup \{(x,\Delta t, e ,\mathcal{R})\}$ \Comment{Concatenate/append.}
        \State $\tau \leftarrow \tau + \Delta t$ \Comment{While loop may terminate here.}
        \State $x=x'$
    \EndWhile
    \State \Return $\mathcal{T}$ \Comment{A set of tuples in preserved order.}
    \EndProcedure
    \end{algorithmic}
\end{algorithm}

\begin{table}[!htbp]
    \centering
    \begin{tabular}{|c|c|c|c|c|c|c|}
    \hline
         \textbf{Distribution} & \textbf{mean} & \textbf{std} & \textbf{min} & \textbf{max} & \textbf{skewness} & \textbf{kurtosis}\\
         \hline
         Average $\hat{f}_{\bar{\nu}}$               &     26.48 &     2.21 &     20.08 &     34.37 &      0.13 &     -0.10 \\
         
         Discounted $\hat{f}_{\nu_\alpha}$           & 14086.51 &  3995.38 &   3735.53 &  32642.98 &      0.38 &      0.08 \\
         
        Blackwell-optimal $\hat{f}_{\nu_{\bar{\alpha}}}$             &  65767.83 &  8148.82 &  38555.27 &  99734.00 &      0.13 &      0.02 \\
        
         Average $\hat{f}_{\bar{\nu}_0}$ (existing)   &      26.54 &     2.30 &     19.34 &     35.29 &      0.12 &      0.07 \\
         
         Discounted $\hat{f}_{\nu_0^\alpha}$ (existing)  &  14142.17 &  3896.38 &   4051.84 &  29226.20 &      0.35 &     -0.04\\
         
        Blackwell-optimal $\hat{f}_{\nu_0^{\bar{\alpha}}}$ (existing)  &  65860.12 &  8215.20 &  39620.38 &  99009.75 &      0.16 &      0.02\\
        \hline
    \end{tabular}
    \caption{Empirical distributions of $\nu$ ($M=5000$).}
    \label{tab:nu_distributions}
\end{table}

\begin{table}[!htbp]
    \centering
    \begin{tabular}{|c|c|c|c|c|c|c|}
    \hline
         \textbf{Distribution} & \textbf{mean} & \textbf{std} & \textbf{min} & \textbf{max} & \textbf{skewness} & \textbf{kurtosis}\\
         \hline
         Average $\hat{f}_{\bar{\eta}}$               &     26.39 &     2.23 &     18.38 &     33.90 &  0.13 & -0.07\\
         
         Discounted $\hat{f}_{\eta_\alpha}$           &  13431.81 &  3733.54 &   3112.43 &  32340.623 &  0.37 &  0.09 \\
         
        Blackwell-optimal $\hat{f}_{\eta_{\bar{\alpha}}}$             &  64729.93 &  7861.13 &  37481.10 &  95706.40 &  0.13 &  0.04 \\
        
         Average $\hat{f}_{\bar{\eta}_0}$ (existing)   &     26.41 &     2.30 &     18.02 &     35.25 &  0.08 &  0.05 \\
         
         Discounted $\hat{f}_{\eta_0^\alpha}$ (existing)  &  13210.41 &  3547.87 &   4802.45&  27309.83 &  0.38 & -0.07 \\
         
        Blackwell-optimal $\hat{f}_{\eta_0^{\bar{\alpha}}}$ (existing)  &  64893.20 &  8232.88 &  37144.30 &  96764.84 &  0.21 & -0.04\\
        \hline
    \end{tabular}
    \caption{Empirical distributions of $\eta$ ($M=5000$).}
    \label{tab:eta_distributions}
\end{table}

\subsection{Significance tests}

The parametric Welch's $t$-test can deal with unequal variances between two distributions but requires them to be normally distributed, as discussed in section~\ref{section: welch}. All distributions have been confirmed to disobey this assumption. However, it is a common misunderstanding that a normality test is an essential prerequisite to be passed for a $t$-test to be used. Furthermore, it can be argued that large samples will fail a normality test as the slightest deviation from normality results in significance. With such a discussion in place, this paper has found \emph{all} of the empirical difference distributions $\Delta \hat{f}$ to obey the normality assumption. These result from using element-wise differences $\Delta \mathbf{C}_f = \mathbf{C}_f \ominus \mathbf{C}_{f_0}$ as samples where $\mathbf{C}_{f_0} $ is the sample array from the existing policy and $\mathbf{C}_f$ pertains to the MDP policy that optimises the studied objective function. This observation is important as it allows for a Student's $t$-test (see section~\ref{section: t test}) to compare $\Delta \hat{f}$ to a hypothesised population mean with all valid assumptions in place. In what follows, this paper tests whether the existing policy is significantly outperformed by an MDP policy, whether it outperforms the MDP policy or if there is no significant difference in which case a new policy need not be adopted. A significance level of $\zeta = 0.05$ has been adopted throughout.

\subsubsection{Existing policy is outperformed}

The null hypothesis for the $t$-test follows
\begin{equation}
    H_0^t: \quad \mathbbm{E}\left[\Delta \hat{f} \right] = 0 \label{eq: t test null hypothesis}
\end{equation}
which can be rejected in favour of an alternative hypothesis
\begin{equation}
    H_A^t:\quad \mathbbm{E}\left[\Delta \hat{f} \right] < 0. \label{eq: t test alternative hypothesis less}
\end{equation}
As to avoid any pairwise correlation resulting from the element-wise subtraction and common random numbers $\omega$, $\mathbf{C}_f$ and $\mathbf{C}_{f_0} $ are randomly shuffled \emph{before} computing $\Delta \mathbf{C}_f $. If correlation remains high, common random numbers should be abandoned. A supplemental non-parametric Mann-Whitney $U$ rank test (see section~\ref{section:mann-whitney}) has been included. However, it does not test for whether the means are the same but assesses the null hypothesis of whether the two samples comes from the same underlying distribution. The alternative hypothesis proposes the MDP policy to be \emph{stochastically smaller} than the distribution of the existing policy. With the understanding that $\hat{F}(x) = \int_{0}^x \hat{f}(z)\, dz$ is an empirical cumulative distribution function, the null hypothesis is given as
\begin{equation}
    H_0^U : \quad \hat{F}_{\pi}(x) = \hat{F}_{\pi_0}(x), \forall x \in \mathbb{R}_{\geq 0} \label{eq: U test null hypothesis}
\end{equation}
along with its alternative 
\begin{equation}
    H_A^U : \quad \hat{F}_{\pi}(x) > \hat{F}_{\pi_0}(x), \forall x \in \mathbb{R}_{\geq 0}.\label{eq: U test alternative hypothesis less}
\end{equation}

These results are presented in table~\ref{tab:nu test less} and \ref{tab:eta test less}. These table also contain the \emph{Pearson correlation coefficients} as well as $p$-values for the Pearson's normality test. These two additional statistics are used to further validate the $t$-test as it provides evidence for the pair-wise sample independent assumption as well as the normality assumption. The results suggest that no MDP policy significantly outperforms the existing policy in terms of either system-based performance metrics.

\begin{table}[!htbp]
    \centering
    \begin{tabular}{|c|c|c|c|c|c|c|c|c|c|c|}
    \hline
        & & & & & &\textbf{reject} & & & \textbf{reject}& \\
        \textbf{Distribution}& $k^2$&  $p$ & \textbf{reject}&$t$ &  $p$ & \textbf{equal}& $U$ &  $p$ & \textbf{same} & \textbf{corr.} \\
        & & & \textbf{normal}& &  & \textbf{means} & & &  \textbf{dist.} &  \\
        \hline
        Average & 0.492 &  0.782 & False & -1.462 & 0.072 &              False & 24139114 &   0.066 &                    False &    0.002 \\
        
         Discounted &  2.709  &  0.258 &         False &  -0.829  & 0.204 &              False & 24266716   & 0.165 &                    False &    -0.012 \\
         
         Blackwell &  1.586   & 0.452 &         False & -0.669    & 0.252 &     False & 24387294 &  0.319 &                    False &    0.005 \\
         \hline
    \end{tabular}
    \caption{Alternative hypothesis: existing policy is outperformed by the MDP policy in terms of $\nu_\pi$.}
    \label{tab:nu test less}
\end{table}

\begin{table}[!htbp]
    \centering
    \begin{tabular}{|c|c|c|c|c|c|c|c|c|c|c|}
    \hline
        & & & & & &\textbf{reject} & & & \textbf{reject}& \\
        \textbf{Distribution}& $k^2$&  $p$ & \textbf{reject}&$t$ &  $p$ & \textbf{equal}& $U$ &  $p$ & \textbf{same} & \textbf{corr.} \\
        & & & \textbf{normal}& &  & \textbf{means} & & &  \textbf{dist.} &  \\
        \hline
        Average & 1.300 &  0.782 &         False & -0.384&  0.072 &              False & 24352359 & 0.066 &                    False &    0.002 \\
        \hline
         Discounted & 2.050&  0.258 &         False &3.601227&  0.204 &              False & 25325127 &  0.165 &                    False &    -0.012 \\
         \hline
         Blackwell &0.704 &  0.452 &         False & -1.198 &  0.252 &     False & 24371546&  0.319 &                    False &    0.005 \\
         \hline
    \end{tabular}
    \caption{Alternative hypothesis: existing policy is outperformed by the MDP policy in terms of $\eta_\pi$.}
    \label{tab:eta test less}
\end{table}

\subsubsection{Existing policy outperforms MDP policies}

Both the $t$-test null hypothesis~(\ref{eq: t test null hypothesis}) and the $U$-test null hypothesis~(\ref{eq: U test null hypothesis}) are retained. The $t$-test replaces (\ref{eq: t test alternative hypothesis less}) with the new alternative hypothesis
\begin{equation}
    H_A^t:\quad \mathbbm{E}\left[\Delta \hat{f} \right] > 0 \label{eq: t test alternative hypothesis greater}
\end{equation}
while the $U$-test forgoes (\ref{eq: U test alternative hypothesis less}) for the alternative of the MDP policy being stochastically larger
\begin{equation}
    H_A^U : \quad \hat{F}_{\pi}(x) < \hat{F}_{\pi_0}(x), \forall x \in \mathbb{R}_{\geq 0}\label{eq: U test alternative hypothesis greater}
\end{equation}
as presented in tables \ref{tab:nu test greater} and \ref{tab:eta test greater}. As the same empirical difference distributions are reused, the correlation and test for normality need not be repeated. A noticeable result is found in table~\ref{tab:eta test greater}. The null hypothesis has been rejected in favour of the existing policy outperforming the more myopic discounted MDP policy with regards to the stationary system-based performance. This is observed in both tests.

\begin{table}[!htbp]
    \centering
    \begin{tabular}{|c|c|c|c|c|c|c|}
    \hline
        \textbf{Dist.} & $t$ & $p$ & \textbf{rej. eq. means}&$U$ & $p$ & \textbf{rej. same dist.} \\
        \hline
        Average &1.463&  0.929 &              False &24139114&  0.934 &                    False \\
         
        Discounted  &-0.823&  0.796&              False &24266716&  0.835 &                    False \\
         
        Blackwell  &-0.669&  0.748 &              False & 24387294& 0.681 &                    False \\
         \hline
    \end{tabular}
    \caption{Alternative hypothesis: existing policy outperforms the MDP policy in terms of $\nu_\pi$.}
    \label{tab:nu test greater}
\end{table}

\begin{table}[!htbp]
    \centering
    \begin{tabular}{|c|c|c|c|c|c|c|}
    \hline
        \textbf{Dist.} & $t$ & $p$ & \textbf{rej. eq. means}&$U$ & $p$ & \textbf{rej. same dist.} \\
        \hline
        Average  &-0.384 &  0.649 &              False & 24352359 &  0.732 &                    False\\
         \hline
        Discounted & 3.601&  $\mathbf{1.59\times 10^{-4}}$ &               \textbf{True} &  25325127 & $\mathbf{2.79\times 10^{-4}}$ &                     \textbf{True} \\
         \hline
        Blackwell&-1.198& 0.885 &              False &24371546&  0.704 &                    False\\
         \hline
    \end{tabular}
    \caption{Alternative hypothesis: existing policy outperforms the MDP policy in terms of $\eta_\pi$.}
    \label{tab:eta test greater}
\end{table}

\subsubsection{Welch's t-test}
The results for the Welch's $t$-test are given separately due to the fact that the normally distributed assumption does not hold for either sample set in its pair. The null hypothesis is given as
\begin{equation}
    H_0^W: \quad  \mathbbm{E}\left[ \hat{f}_{\pi}  \right] = \mathbbm{E}\left[  \hat{f}_{\pi_0} \right]
\end{equation}
with an alternative hypothesis taking form of either the MDP policy being superior
\begin{equation}
    H_{A_1}^W: \quad  \mathbbm{E}\left[ \hat{f}_{\pi}  \right] < \mathbbm{E}\left[  \hat{f}_{\pi_0} \right] \label{eq: alt 1}
\end{equation}
or the existing policy being superior
\begin{equation}
    H_{A_2}^W: \quad  \mathbbm{E}\left[ \hat{f}_{\pi}  \right] > \mathbbm{E}\left[  \hat{f}_{\pi_0} \right].\label{eq: alt 2}
\end{equation}

\begin{table}[!htbp]
    \centering
    \begin{tabular}{|c|c|c|c|c|c|c|}
    \hline
        \textbf{Dist.} &$t$& $p$ & \textbf{rej. for $A_1$ (\ref{eq: alt 1})}& $t$&$p$ & \textbf{rej. for $A_2$ (\ref{eq: alt 2})} \\
        \hline
        Average  & -1.461&  0.072 &        False & -1.461& 0.928 &        False\\
         
        Discounted  &-0.834 & 0.202 &        False & -0.834&  0.798 &        False\\
         
        Blackwell  &-0.667& 0.252 &        False & -0.667& 0.748&        False\\
         \hline
    \end{tabular}
    \caption{Welch's $t$-test for $\nu_\pi$.}
    \label{tab:welch test nu}
\end{table}

\begin{table}[!htbp]
    \centering
    \begin{tabular}{|c|c|c|c|c|c|c|}
    \hline
        \textbf{Dist.} &$t$& $p$ & \textbf{rej. for $A_1$ (\ref{eq: alt 1})}& $t$&$p$ & \textbf{rej. for $A_2$ (\ref{eq: alt 2})} \\
        \hline
        Average  &-0.381&  0.352 &        False &-0.381&  0.648 &        False \\
         
        Discounted  &3.596 & 0.999 &        False &3.596&   $\mathbf{1.62\times 10^{-4}}$ &         \textbf{True}\\
         
        Blackwell  &-1.120& 0.115 &        False &-1.120&  0.885 &        False \\
         \hline
    \end{tabular}
    \caption{Welch's $t$-test for $\eta_\pi$.}
    \label{tab:welch test eta}
\end{table}

The $p$-values vaguely differ from those obtained through the Student's $t$-test. The same outcome is obtained: evidence exist to support the discounted MDP policy as inferior to the exiting policy based the stationary system-based performance.

\newpage

\section{Conclusion}
This paper has proposed four types of scalar system-based performance metrics that can be used to gauge how effective policies are. Such scalars are appealing as they allow for simulations to be used in constructing empirical uni-variate distributions. These are used in determining whether a policy provides a \emph{statistically significant} advantage as opposed to a theoretical one. Furthermore, such distributions can be obtained from a simulator that is much more general and complex than the MDP model under which the policy was solved for. It has been shown that state-based optimality, the foundation in solving for an optimal MDP policy, does not always lead to system-based optimality as was the case with $\eta_{\pi}^\alpha$. Such system-based optimality was only guaranteed to hold for Blackwell optimal policies that induce a uni-chain. Moreover, system-based optimality does not always yields unique optima for $\eta_{\pi}$, $\nu_{\pi}$ and $\eta_{\pi}^\alpha$ if the set of gain-optimal policies is larger than one $|\Pi_2| >1$. 

The system-based performance metric relies heavily on an initial distributions over states. Further work should investigate different distributions such as (\ref{eq: hybrid metric}). With regards to the hypothesis testing, modern $A/B$ testing should also be considered. Future work may also focus on other performance statistics that can be derived from the empirical distributions other than its mean. This may include variance and/or Value-at-Risk (VaR). All results obtained here hold for finite-state and finite-action MDPs. The results of this paper should be extended to cases where these are infinite/continuous. Lastly, results have not been formulated for optimal policies that induce a multi-chain.

\newpage
\medskip
\printbibliography

\end{document}